\edef\savecatcodeat{\the\catcode`@}
\catcode`\@=11

\def\tb@ifSpecChars#1#2{#1}
\def\tb@ifNoSpecChars#1#2{#2}

\def\tableau{%
  \bgroup
  \@ifstar{\let\Tif\tb@ifNoSpecChars\tb@tableauB}
          {\let\Tif\tb@ifSpecChars\tb@tableauB}}

\def\tb@tableauB{
  \@ifnextchar[{\tb@tableauC}{\tb@tableauC[]}}

\def\tb@tableauC[#1]{\hbox\bgroup%
    \let\\=\cr
    \def\bl{\global\let\tbcellF\tb@cellNF}%
    \def\tf{\global\let\tbcellF\tb@cellH}
    \dimen2=\ht\strutbox \advance\dimen2 by\dp\strutbox%
    \ifx\baselinestretch\undefined\relax%
    \else%
       \dimen0=100sp \dimen0=\baselinestretch\dimen0%
       \dimen2=100\dimen2 \divide\dimen2 by\dimen0%
    \fi%
    \let\tpos\tb@vcenter
    \tb@initYoung
    \tb@options#1\eoo
    \let\arrow\tb@arrow%
    \dimen0=\Tscale\dimen2%
    \dimen1=\dimen0 \advance\dimen1 by \tb@fframe%
    \lineskip=0pt\baselineskip=0pt
    \def\tb@nothing{}%
    \def\endcellno{$\rss\egroup\bss\egroup}
    \def\endcell{\endcellno\kern-\dimen0}
    \def\begincell{\vbox to\dimen0\bgroup\vss\hbox to\dimen0\bgroup\hss$}%
    \let\overlay\tb@overlay%
    \let\fl\tb@fl%
    \let\lss\hss\let\rss\hss\let\tss\vss\let\bss\vss
    \def\mkcell##1{
        \let\tbcellF\tb@cellD
        \def\tb@cellarg{##1}
        \ifx\tb@cellarg\tb@nothing\let\tb@cellarg\tb@cellE\fi%
        \begincell\tb@cellarg\endcellno
        \tbcellF}
    \let\savecellF\tbcellF
     \Tif{\catcode`,=4\catcode`|=\active}{}\tb@tableauD}%

\let\tb@savetableauD\tableauD
{
    \catcode`|=\active \catcode`*=\active \catcode`~=\active%
    \catcode`@=\active
\gdef\tableauD#1{%
  \Tif{
    \mathcode`|="8000 \mathcode`*="8000%
    \mathcode`~="8000 \mathcode`@="8000%
    \def@{\bullet}%
    \let|\cr
    \let*\tf
    \let~\sk
  }{}%
  \tpos{\tabskip=0pt\halign{&\mkcell{##}\cr#1\crcr}}%
  \global\let\tbcellF\savecellF
  \egroup
  \egroup}
}
\let\tb@tableauD\tableauD
\let\tableauD\tb@savetableauD
\let\tb@savetableauD\undefined

\def\tb@options#1{\ifx#1\eoo\relax\else\tb@option#1\expandafter\tb@options\fi}

\def\tb@option#1{%
  \if#1t\let\tpos\tb@vtop\fi
  \if#1c\let\tpos\tb@vcenter\fi
  \if#1b\let\tpos\vbox\fi
  \if#1F\tb@initFerrers\fi
  \if#1Y\tb@initYoung\fi
  \if#1s\tb@initSmall\fi
  \if#1m\tb@initMedium\fi
  \if#1l\tb@initLarge\fi
  \if#1p\tb@initPartition\fi
  \if#1a\tb@initArrow\fi
}

\def\tb@vcenter#1{\ifmmode\vcenter{#1}\else$\vcenter{#1}$\fi}

\def\tb@vtop#1{\hbox{\raise\ht\strutbox\hbox{\lower\dimen0\vtop{#1}}}}

\def\tb@initPartition{\def\Tscale{.3}}
\def\tb@initSmall{\def\Tscale{1}}
\def\tb@initMedium{\def\Tscale{2}}
\def\tb@initLarge{\def\Tscale{3}}

\def\tb@initArrow{\dimen2=1.25em}

\def\tb@initYoung{%
  \def\tb@cellE{}
  \let\tb@cellD\tb@cellN
  \def\sk{\global\let\tbcellF\tb@cellNF}}
\def\tb@initFerrers{%
  \def\tb@cellE{\bullet}
  \let\tb@cellD\tb@cellNF
  \def\sk{\bullet}}

\tb@initMedium

\def\tb@sframe#1{%
  \vbox to0pt{
    \vss
    \hbox to0pt{%
      \hss
      \vbox to\dimen1{
        \hrule depth #1 height0pt
        \vss
        \hbox to\dimen1{
          \vrule width #1 height\dimen1
          \hss
          \vrule width #1
          }%
        \vss
        \hrule height #1 depth 0in
        }%
      \kern-\tb@hframe
      }%
    \kern-\tb@hframe}}

\def\tb@hframe{.2pt}\def\tb@fframe{.4pt}\def\tb@bframe{2pt}
\def\tb@cellH{\tb@sframe{\tb@bframe}}       
\def\tb@cellNF{}                            
\def\tb@cellN{\tb@sframe{\tb@fframe}}       
\let\tbcellF\tb@cellN                       

\def\tb@rpad{1pt}
\def\tb@lpad{1pt}
\def\tb@tpad{1.8pt}
\def\tb@bpad{1.8pt}

\def\tb@overlay{\endcell\@ifnextchar[{\tb@overlaya}{\begincell}}
\def\tb@overlaya[#1]{\vbox to\dimen0\bgroup%
  \tb@overlayoptions#1\eoo%
  \tss\hbox to\dimen0\bgroup\lss$}
\def\tb@overlayoptions#1{\ifx#1\eoo\relax\else\tb@overlayoption#1\expandafter\tb@overlayoptions\fi}

\def\tb@overlayoption#1{
  \if#1t\def\tss{\vskip\tb@tpad}\let\bss\vss\fi
  \if#1c\let\tss\vss\let\bss\vss\fi
  \if#1b\def\bss{\vskip\tb@bpad}\let\tss\vss\fi
  \if#1l\def\lss{\hskip\tb@lpad}\let\rss\hss\fi
  \if#1m\let\lss\hss\let\rss\hss\fi
  \if#1r\def\rss{\hskip\tb@rpad}\let\lss\hss\fi
}

\def\tb@fl{\endcell\begincell\vrule depth 0pt width \dimen0 height \dimen0 \endcell\begincell}

\@ifundefined{diagram}{}{
\def\tb@arrowpad{.5}

\newoptcommand{\tb@arrow}{\@ne}[2]{%
  \endcell
   \begingroup%
   \let\dg@getnodesize\tb@getnodesize
   \dg@USERSIZE=#1\relax%
   \ifnum\dg@USERSIZE<\@ne \dg@USERSIZE=\@ne \fi%
   \dg@parse{#2}%
   \dg@label{\tb@draw{#1}{#2}}}

\def\tb@getnodesize#1#2#3#4#5{\dimen3=\tb@arrowpad\dimen2 #4=\dimen3 #5=\dimen3\relax}
\def\tb@getnodesize#1#2#3#4#5{\ifnum#2=0\ifnum#3=0\tb@getnodesizetail{#4}{#5}\else\tb@getnodesizehead{#4}{#5}\fi\else\tb@getnodesizehead{#4}{#5}\fi}
\def\tb@getnodesizetail#1#2{\dimen3=.5\dimen2 #1=\dimen3 #2=\dimen3}
\def\tb@getnodesizehead#1#2{\dimen3=.5\dimen2 #1=\dimen3 #2=\dimen3}

\def\tb@draw#1#2#3#4{%
        \dg@X=0\dg@Y=0\dg@XGRID=1\dg@YGRID=1\unitlength=.001\dimen0%
        \dg@LBLOFF=\dgLABELOFFSET \divide\dg@LBLOFF\unitlength%
        \dg@drawcalc
        \begincell
        \let\lams@arrow\tb@lams@arrow
        \begin{picture}(0,0)\begingroup\dg@draw{#1}{#2}{#3}{#4}\end{picture}%
        \endcell
        \endgroup
        \begincell}
}

\def\tb@lams@arrow#1#2{%
 \lams@firstx\z@\lams@firsty\z@
 \lams@lastx#1\relax\lams@lasty#2\relax
 \lams@center\z@
 \N@false\E@false\H@false\V@false
 \ifdim\lams@lastx>\z@\E@true\fi
 \ifdim\lams@lastx=\z@\V@true\fi
 \ifdim\lams@lasty>\z@\N@true\fi
 \ifdim\lams@lasty=\z@\H@true\fi
 \NESW@false
 \ifN@\ifE@\NESW@true\fi\else\ifE@\else\NESW@true\fi\fi
 \ifH@\else\ifV@\else
  \lams@slope
  \ifnum\lams@tani>\lams@tanii
   \lams@ht\ten@\p@\lams@wd\ten@\p@
   \multiply\lams@wd\lams@tanii\divide\lams@wd\lams@tani
  \else
   \lams@wd\ten@\p@\lams@ht\ten@\p@
   \divide\lams@ht\lams@tanii\multiply\lams@ht\lams@tani
  \fi
 \fi\fi
 \ifH@  \lams@harrow
 \else\ifV@ \lams@varrow
 \else \lams@darrow
 \fi\fi
}

\catcode`\@=\savecatcodeat
\let\savecatcodeat\undefined

\documentclass[tbtags,reqno]{amsart}

\usepackage[tbtags]{amsmath}
\usepackage{amsopn,amsthm,amsfonts,amssymb}

\newcommand{\tcercle}[1]{\ensuremath{\setlength{\unitlength}{1ex}\begin{picture}(2.8,2.8)\put(1.4,1.4){\circle{2.8}\makebox(-5.6,0){#1}}\end{picture}}}
\newcommand{\gcercle}{\ensuremath{\setlength{\unitlength}{1ex}\begin{picture}(5,5)\put(2.5,2.5){\circle{5}}\end{picture}}}

\def\La{\Lambda}
\let\la\lambda
\let\ta\theta

\let\Om\Omega

\catcode`\Ä=13 \def Ä{\"A} \catcode`\Å=13 \def Å{\AA }
\catcode`\Ç=13 \def Ç{\c C} \catcode`\É=13 \def É{\'E}
\catcode`\Ñ=13 \def Ñ{\~N} \catcode`\Ö=13 \def Ö{\"O}
\catcode`\Ü=13 \def Ü{\"U} \catcode`\á=13 \def á{\'a}
\catcode`\à=13 \def à{\`a} \catcode`\â=13 \def â{\^a}
\catcode`\ä=13 \def ä{\"a} \catcode`\ã=13 \def ã{\~a}
\catcode`\å=13 \def å{\aa } \catcode`\ç=13 \def ç{\c c}
\catcode`\é=13 \def é{\'e} \catcode`\è=13 \def è{\`e}
\catcode`\ê=13 \def ê{\^e} \catcode`\ë=13 \def ë{\"e}
\catcode`\í=13 \def í{\'\i } \catcode`\ì=13 \def ì{\`\i }
\catcode`\î=13 \def î{\^\i } \catcode`\ï=13 \def ï{\"\i }
\catcode`\ñ=13 \def ñ{\~n} \catcode`\ó=13 \def ó{\'o}
\catcode`\ò=13 \def ò{\`o} \catcode`\ô=13 \def ô{\^o}
\catcode`\ö=13 \def ö{\"o} \catcode`\õ=13 \def õ{\~o}
\catcode`\ú=13 \def ú{\'u} \catcode`\ù=13 \def ù{\`u}
\catcode`\û=13 \def û{\^u} \catcode`\ü=13 \def ü{\"u}
\catcode`\§=13 \def §{\S }
\catcode`\¶=13 \def ¶{\P } \catcode`\ƒ=13 \def ƒ{\infty}
\catcode`\ß=13 \def ß{\ss } \catcode`\Æ=13 \def Æ{\AE }
\catcode`\Ø=13 \def Ø{\O } \catcode`\æ=13 \def æ{\ae }
\catcode`\ø=13 \def ø{\o } \catcode`\¿=13 \def ¿{{?`}}
\catcode`\«=13 \def «{\ll } \catcode`\»=13 \def »{\gg }
\catcode`\Š=13 \def Š{\dots } \catcode`\À=13 \def À{\`A}
\catcode`\Ã=13 \def Ã{\~A} \catcode`\Õ=13 \def Õ{\~O}
\catcode`\'=13 \def '{\OE } \catcode`\¦=13 \def ¦{\oe }
\catcode`\"=13 \def "{{``}} \catcode`\"=13 \def "{{''}}
\catcode`\¬=13 \def ¬{\hfill\break}

\def«{\raise 0.1em\hbox{$\scriptscriptstyle\langle\kern -0.3em\langle$} }
\def»{\hskip 0.2em\raise 0.1em\hbox{$\scriptscriptstyle\rangle\kern
-0.3em\rangle$} }

\newcommand{\LL}{\ensuremath{\langle\!\langle}}
\newcommand{\RR}{\ensuremath{\rangle\!\rangle}}

\newtheorem{theorem}{Theorem}
\newtheorem{lemma}[theorem]{Lemma}
\newtheorem{proposition}[theorem]{Proposition}

\newtheorem{identity}[theorem]{Identity}

\theoremstyle{definition}

\theoremstyle{remark}
\newtheorem{remark}[theorem]{Remark}
\newtheorem*{acknow}{Acknowledgments}

\usepackage{graphicx}
\usepackage{psfrag}

\def\la{\lambda}
\def\ga{\gamma}
\renewcommand{\geq}{\geqslant}
\renewcommand{\leq}{\leqslant}

\newcommand{\Par}[3]{P_{#1,#2}^{[#3]}}

\newcommand{\Parsh}[3]{P_{#1,#2}^{[#3],+}}
\newcommand{\PPar}[3]{\mathcal{P}_{#1,#2}^{[#3]}}
\newcommand{\EPPar}[3]{\mathcal{E}\mathcal{P}_{#1,#2}^{[#3]}}
\newcommand{\Del}[3]{\Delta_{#1,#2}^{[#3]}}
\newcommand{\Delsh}[3]{\Delta_{#1,#2}^{[#3],+}}

\begin{document}

\title[A normalization
formula for the Jack polynomials in superspace and an identity on partitions]
{A normalization
formula for the Jack polynomials in superspace and an identity on partitions}

\author{Luc Lapointe}  \thanks{L. L. was partially supported by the
Anillo Ecuaciones Asociadas a Reticulados financed by the World Bank
through the Programa Bicentenario de Ciencia y Tecnolog\'{\i}a, and by
the Programa Reticulados y Ecuaciones of the Universidad de
Talca.}
\address{Instituto de Matem\'atica y F\'{\i}sica,
Universidad de Talca, Casilla 747, Talca, Chile}
\email{lapointe@inst-mat.utalca.cl}
\author{Yvan Le Borgne}
\thanks{Y.L.B. was partially supported by the French Agence Nationale de la Recherche, 
projects SADA ANR-05-BLAN-0372 and MARS ANR-06-BLAN-0193.
}
\address{CNRS, Universit\'e de Bordeaux, LaBRI, 351 Cours de la
Lib\'eration, 33405 Talence Cedex, France}
 \email{yvan.leborgne@labri.fr}
\author{Philippe Nadeau}
\thanks{P.N. was supported by the Austrian
Science Foundation FWF, grant S9607-N13, in the framework of the
National Research Network ``Analytic Combinatorics and Probabilistic
Number Theory''.}
\address{Fakult\"at f\"ur Mathematik, Universit\"at Wien,
    Nordbergstrasse 15, 1090 Vienna, Austria}
\email{philippe.nadeau@univie.ac.at}


\begin{abstract}   We prove a conjecture of \cite{DLM2} giving
a closed form formula for the
norm of the Jack polynomials in superspace with respect to a certain
scalar product.  The proof is mainly combinatorial and relies on the
explicit expression in terms of admissible tableaux of the non-symmetric
Jack polynomials.  In the final step of the proof appears
an identity on weighted sums of partitions that we demonstrate
using the methods of Gessel-Viennot.
\end{abstract}


\maketitle

\section{Introduction}

Let
$(x,\theta)=(x_1, \cdots x_N,\ta_1, \cdots \ta_N)$ be a collection of $2N$
variables, called respectively bosonic and fermionic (or
anticommuting or Grassmannian), obeying the relations
\begin{equation}
x_ix_j=x_jx_i, \, \qquad
x_i\ta_j=\ta_jx_i\, \qquad {\rm and} \qquad \ta_i\ta_j=-\ta_j\ta_i \qquad (\Rightarrow \theta_i^2=0)\;.
\end{equation}
We call symmetric functions in superspace the
ring of
polynomials in these variables
over the field $\mathbb Q$ that are
invariant  under the simultaneous interchange of $x_i
\leftrightarrow x_j$ and $\theta_i \leftrightarrow \theta_j$ for any
$i,j$.  That is, defining
\begin{equation}
\mathcal K_{\sigma} f(x_1,\dots,x_N,\theta_1,\dots,\theta_N):=
f(x_{\sigma(1)},\dots,x_{\sigma(N)},\theta_{\sigma(1)},\dots,\theta_{\sigma(N)})
\, ,\qquad \sigma \in S_N
\, ,
\end{equation}
we have that a polynomial $f(x_1,\dots,x_N,\theta_1,\dots,\theta_N)$
is a symmetric function in superspace iff
\begin{equation}
\mathcal K_{\sigma} f(x_1,\dots,x_N,\theta_1,\dots,\theta_N)
= f(x_1,\dots,x_N,\theta_1,\dots,\theta_N)
\end{equation}
for all permutations $\sigma$ in the symmetric group $S_N$.

Bases of the ring of symmetric functions in superspace  can be indexed by
superpartitions.  A superpartition $\Lambda$ is of the form
\begin{equation}
\Lambda:=(\La^a;\La^s)=(\Lambda_1,\ldots,\Lambda_m;\Lambda_{m+1},\ldots,\Lambda_{N})\,
,\end{equation}
where
\begin{equation}
\Lambda_1 > \Lambda_2 > \cdots > \Lambda_m \geq 0 \qquad {\rm and}
\qquad \Lambda_{m+1} \geq \Lambda_{m+2} \geq \cdots \geq \Lambda_N \geq 0 \, .
\end{equation}
In other words,
$\La^a$ is a partition with distinct parts
(one of them possibly equal to zero), and $\La^s$ is an ordinary
partition.
The degree of
$\Lambda$ is
$|\Lambda|=\Lambda_1+\cdots+\Lambda_N$ while its fermionic degree is
$m$.  The length $\ell(\Lambda)$ of $\Lambda$ is $m+\ell(\Lambda^s)$, where
$\ell(\Lambda^s)$ is the number of non-zero parts in the partition $\Lambda^s$
(the usual length of a partition).
Given a fixed degree $n$ and fermionic degree $m$,
a superpartition that will be especially relevant for this work is
\begin{equation}\label{minide}
\Lambda_{\mathrm{min}}:=(\delta_m\,;\,
1^{\ell_{n,m}}\,) \, ,
\end{equation} where
\begin{equation}
\delta_m:=(m-1,m-2,\ldots,0)  \quad {\rm and} \quad
\ell_{n,m}:=n-\frac{m(m-1)}{2}
\, .
\end{equation}
The superpartition
$\Lambda_{\mathrm{min}}$ is the minimal one among the superpartitions of
 degree $n$ and fermionic degree $m$ in some order on superpartitions
 generalizing
the dominance order on partitions (see \cite{DLM2}).  Note that it will
always be clear from the context what $n$ and $m$ are.

A natural basis for the ring of symmetric functions in superspace
is given by the monomial functions:
\begin{equation} m_{\Lambda}= \frac{1}{f_{\Lambda^s}}
\sum_{\sigma\in S_{N}} {\mathcal K}_{\sigma} \, \theta_1 \cdots
  \theta_m \, x^{\Lambda} \, ,
\end{equation}
where
\begin{equation}
x^{\Lambda} :=
x_1^{\Lambda_1} \cdots x_{m}^{\Lambda_m} x_{m+1}^{\Lambda_{m+1}} \cdots x_{N}^{\Lambda_N}
\end{equation}
and
\begin{equation} \label{equadeff}
f_{\Lambda^s} = \prod_{i\geq 0} m_i({\Lambda^s})! \, ,
\end{equation}
with $m_i({\Lambda^s})$ the number of $i$'s in the partition $\Lambda^s$.

A less trivial basis of the the ring of symmetric functions in
superspace is given by the Jack polynomials in superspace,
$J_{\Lambda}$, which generalize the usual Jack polynomials.
These polynomials, depending on a parameter $\alpha$, arose as
eigenfunctions of a supersymmetric quantum-mechanical many-body problem.
An explicit definition of the Jack polynomials in
superspace involving
non-symmetric Jack polynomials will be given in
Section~\ref{subsjack}.

The main point of this article is to prove a conjecture,
stated in \cite{DLM2},
giving an explicit expression for
the coefficient $c_{\Lambda}^{\mathrm{min}}(\alpha)$
of $\tilde m_{\Lambda_{\mathrm{min}}} :=(\ell_{n,m}!) m_{\Lambda_{\mathrm{min}}}$ in
$J_{\Lambda}$, where $n=|\Lambda|$ and $m$ is
the fermionic degree of $\Lambda$ (see Proposition~\ref{propo}).
The relevance of this conjecture is that it gives
as a corollary an explicit form for the norm of the Jack polynomials in superspace
with respect to a certain scalar product.  To be more precise,
for a superpartition $\Lambda$,
let the corresponding
power sum products in superspace be given
by
\begin{equation}
p_\La:=\tilde{p}_{\La_1}\ldots\tilde{p}_{\La_m}p_{\La_{m+1}}\cdots
p_{\La_N}\quad\text{with}\qquad
p_n:=m_{(;n)}\quad\text{and}\quad\tilde{p}_k:=m_{(k;0)}\,
,\end{equation} and define
the scalar product:
\begin{equation} \label{defscalprodcomb} \LL
\, {p_\La} \, | \, {p_\Om } \, \RR_{\alpha}:=(-1)^{m(m-1)/2}z_\La
(\alpha)\delta_{\La,\Om}\,,\qquad
z_\La(\alpha):=\alpha^{{\ell}(\La)} \prod_{i \geq 1} i^{m_i(\Lambda^s)} m_i(\Lambda^s)!
\, .
\end{equation}
As shown in \cite{DLM2}, the Jack polynomials in superspace are such that
\begin{equation}
 \LL
\, {J_\La} \, | \, {J_\Om } \, \RR_{\alpha} = \alpha^{m+\ell_{n,m}}
\frac{c_{\Lambda}^{\mathrm{min}}(\alpha)}{c_{\Lambda'}^{\mathrm{min}}(1/\alpha)}
\, \delta_{\Lambda,\Omega}\, ,
\end{equation}
where $\Lambda'$, the conjugate of $\Lambda$, will be described at the
end of Section~\ref{subs21}. Obtaining an explicit expression for
$c_{\Lambda}^{\mathrm{min}}(\alpha)$ thus immediately gives a closed form for
the norm of the Jack polynomials in superspace with respect to this scalar
product.  We should point out that these results are natural analogs of classical
results on Jack polynomials (see for instance \cite{Mac}).

The proof of Proposition~\ref{propo} relies on the
explicit expressions for non-symmetric Jack polynomials in terms
of admissible tableaux given in \cite{sahi}.  An interesting by-product of the proof
is that it leads to an identity on partitions (see Identity~\ref{iden2})
that we believe is worth stating here in the special case $\gamma=0^{m-1}$.

\begin{identity} \label{iden20}
For $i=1,\dots,m$,
let $\lambda^{(i)}$ be a partition of length $i$ with no parts
larger than $m$.
We say that $\lambda^{(1)},\dots,\lambda^{(m)}$ are non-intersecting
if  the $j$-th parts of $\lambda^{(j)},\lambda^{(j+1)},\dots,\lambda^{(m)}$
are distinct for $j=1,\dots,m$.  In particular, this implies that
$[\lambda^{(1)}_1,\dots,\lambda^{(m)}_1]$ is a permutation in $S_m$.
We define
$\mathcal V_0$ to be the set of
$(\lambda^{(1)},\dots,\lambda^{(m)})$ such that
$\lambda^{(1)},\dots,\lambda^{(m)}$ are non-intersecting.
We say that $(i,j)$ is critical in
$(\lambda^{(1)},\dots,\lambda^{(m)}) \in \mathcal V_0$ if
$i \geq j \geq 2$ and $\lambda^{(i)}_j=\lambda^{(i)}_{j-1}$.
Finally,  let
$a_1,\dots,a_m$ and $b_1,\dots,b_{m-1}$
be indeterminates.  We have
\begin{equation} \label{idenequa}
\prod_{1\leq j < i \leq m} (a_i+1-a_j)
=\sum_{(\lambda^{(1)},\dots,\lambda^{(m)}) \in \mathcal V_0}
{\rm sgn}([\lambda_1^{(1)},\dots,\lambda_1^{(m)}]) \,
\prod_{(i,j)~{\rm critical}} (a_{\lambda^{(i)}_j} +
b_{j-1}) \, .
\end{equation}
\end{identity}
Observe that the L.H.S. does not depend on the $b_i$'s while the R.H.S. does.
The proof we provide of this identity
relies crucially on the identification of the R.H.S. of \eqref{idenequa}
as a determinant using the methods of
Gessel-Viennot \cite{LGV}.

\section{Definitions}

\subsection{Superpartitions} \label{subs21}
Superpartitions were defined in the introduction.  We describe here a
diagrammatic representation of superpartitions that
extends the notion of Ferrers' diagram.  Recall \cite{Mac}
that the Ferrers' diagram
of the partition $\lambda=(\lambda_1,\dots,\lambda_r)$
is the set of cells in $\mathbb Z^2_{\geq 1}$ such that
$1 \leq i \leq r$ and $1 \leq j \leq \lambda_i$.  We use here the convention
in which $i$ increases as one goes down.  For instance, to
$\lambda=(5,3,1,1)$ corresponds the diagram
\begin{equation}
{\tableau[scY]{&&&& \\ & &\\ \\ \\  }}
\end{equation}

To every
superpartition $\Lambda$, we can associate a unique partition $\Lambda^*$
obtained by deleting the semicolon and reordering the parts in
non-increasing order.  The diagram associated to
$\Lambda$, denoted by $D[\Lambda]$,
is obtained by first drawing the Ferrers' diagram associated to
$\Lambda^*$ and then adding a circle at the end of each row corresponding
to an entry of
$\Lambda^a$.
If an entry  of $\Lambda^a$ coincides with some entries
of $\Lambda^s$, the row corresponding to that entry in $D[\Lambda]$
is considered to be the topmost one.
For instance, if
$\Lambda=(3,1,0;5,3,2)$, we have $\Lambda^*=(5,3,3,2,1,0)$, and thus
\begin{equation} \label{exdia}D([3,1,0;5,3,2]
  )={\tableau[scY]{&&&& \\ &&&\bl\tcercle{}\\ && \\ &  \\ & \tcercle{}\bl
\\\tcercle{}\bl\\ }}
\end{equation}
Note that with this definition, if the circles are considered as cells then
$D[\Lambda]$ is still a partition.  It is thus natural to define $\Lambda'$,
the conjugate of $\Lambda$, to be the superpartition obtained by transposing
the diagram of $D[\Lambda]$ with respect to the main diagonal.  Using
the example above, one easily sees that $(3,1,0;5,3,2)'=(5,4,1;3,1)$.

\subsection{Non-symmetric Jack polynomials}

The non-symmetric Jack polynomials were first studied in \cite{Op} (although
they had appeared before in physics as eigenfunctions of certain Dunkl-type operators
\cite{Pas}).
These
are polynomials $E_{\eta}(x;\alpha)$ in a given number $N$ of variables
$x=x_1,\dots,x_N$, depending on
a formal parameter $\alpha$ and indexed by compositions.
For our purposes, we will reproduce
the explicit combinatorial formula given in \cite{sahi}.
Let $\eta \in \mathbb Z^N_{\geq 0}$ be a composition with $N$ parts
(some of them possibly equal to zero).  The diagram of
$\eta$ is the set of cells in $\mathbb Z^2_{\geq 1}$ such that
$1 \leq i \leq N$ and $1 \leq j \leq \eta_i$.  For instance,
if $\eta=(0,1,3,0,0,6,2,5)$, the diagram of $\eta$ is
\begin{equation}
{\tableau[scY]
{\bl \bullet \\  \\ && \\ \bl \bullet \\ \bl \bullet \\  & & & & &  \\
 & \\
& & & &  \\ }}
\end{equation}
where a $\bullet$ represents an entry of length zero.   For each cell
$s=(i,j) \in \eta$, we define its arm-length ${\bf a}_{\eta}(s)$, leg-length
${\mathbf l}_{\eta}(s)$ and  $\alpha$-hooklength $d_{\eta}(s)$ by:
\begin{eqnarray} \nonumber
{\bf a}_{\eta}(s) & = & \eta_i -j \\ \nonumber
{\bf l}_{\eta}'(s) & = &
\# \{ k=1,\dots,i-1 \, | \, j \leq \eta_k+1 \leq \eta_i  \}  \\ \nonumber
{\bf l}_{\eta}''(s) & = &
\# \{ k=i+1,\dots,N \, | \, j \leq \eta_k \leq \eta_i  \} \\\nonumber
{\bf l}_{\eta}(s) & = & {\bf l}_{\eta}'(s) +
{\bf l}_{\eta}''(s) \\ \nonumber
d_{\eta}(s) & = &  \alpha( {\bf a}_{\eta}(s)+1) + {\bf l}_{\eta}(s)+1.
\end{eqnarray}
A diagrammatic representation of these parameters is provided in
Figure~\ref{fig:alphahooklengthdefinition}.
\begin{figure}
\centering
\psfrag{1}{$1$}
\psfrag{a}{$\alpha$}
\psfrag{+1}{{+}$1$}
\psfrag{i}{$i$}
\psfrag{j}{$j$}
\psfrag{Lprime}{$L'(i,j)$}
\psfrag{Ldoubleprime}{$L''(i,j)$}
\psfrag{lprime}{${\bf l}_{\eta}'(i,j)=3$}
\psfrag{ldoubleprime}{${\bf l}_{\eta}''(i,j)=4$}
\includegraphics[height=8cm]{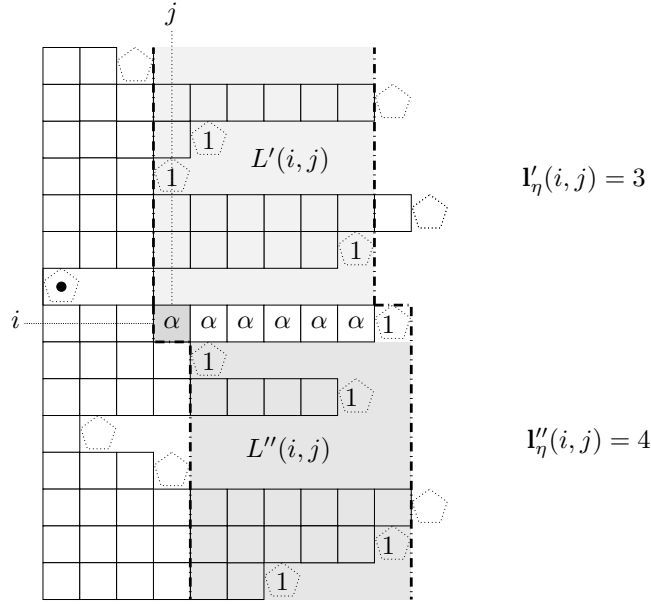}
\caption{Diagrammatic representation of the $\alpha$-hooklength of the
  cell $s=(i,j)=(8,4)$. We add a (dotted) pentagonal cell at the end
  of each row. The three terms $1+{\bf l}_{\eta}'(s)+{\bf
  l}_{\eta}''(s)$ of the $\alpha$-hook length count respectively the
  pentagonal cell of row $i$, the number of pentagonal cells that
  belong to the set $L'(s) = \{(k,l) \, | \,  k < i \mbox{ and } j \leq l
  \leq \eta_i\}$ and the number of pentagonal cells that belong to
  $L''(s)= \{(k,l) \, | \, i < k \mbox{ and } j+1 \leq l \leq \eta_i+1 \}
  $. The coefficient ${\bf a}_{\eta}(s)+1$ of $\alpha$ counts the
  cells in row $i$ from $(i,j)$ to $(i,\eta_i)$.  In this example we
  have $d_{\eta}(s) = (1+3+4)+6\alpha$.
\label{fig:alphahooklengthdefinition}
}
\end{figure}
An explicit formula for $E_{\eta}(x;\alpha)$ is given in terms of certain
tableaux called
$0$-admissible tableaux.  A 0-admissible tableau $T$ of shape $\eta$
is a filling of the cells of $\eta$ with letters belonging to
 $\{1,2, \dots, N \}$
satisfying the following properties:
\begin{enumerate}
\item[(1)]  There are never two identical letters in the same column;
\item[(2)] If the cell $(i,j)$ is filled with letter $c$, then
a letter $c$ cannot occur in column $j+1$ in a row below row $i$;
\item[(3)] In the first column, a letter $i$ cannot occur in a row
below row $i$.
\end{enumerate}
A cell $(i,j)$ in a 0-admissible
tableau is called 0-critical if either:
\begin{enumerate}
\item[(a)] $j>1$ and cell $(i,j-1)$
is filled with the same letter as cell $(i,j)$
\item[(b)] $j=1$ and cell $(i,j)=(i,1)$ is filled with letter $i$.
\end{enumerate}
\begin{remark} As observed in \cite{sahi},
conditions (3) and (b) can be made superfluous if one defines
a tableau $T^0$ obtained from $T$ by adding a column $0$ filled with an $i$ in
row $i$ for $i=1,\dots,N$.  In this case $T$ is $0$-admissible if $T^0$
satisfies (1) and (2).  And $s$ is $0$-critical if it satisfies (a) when
considered in $T^0$.
\end{remark}

\begin{figure}
\psfrag{1}{$1$}
\psfrag{2}{$2$}
\psfrag{3}{$3$}
\psfrag{4}{$4$}
\psfrag{5}{$5$}
\psfrag{6}{$6$}
\psfrag{7}{$7$}
\psfrag{8}{$8$}
\psfrag{9}{$9$}
\psfrag{empty}{$\bullet$}
\includegraphics[height=4cm]{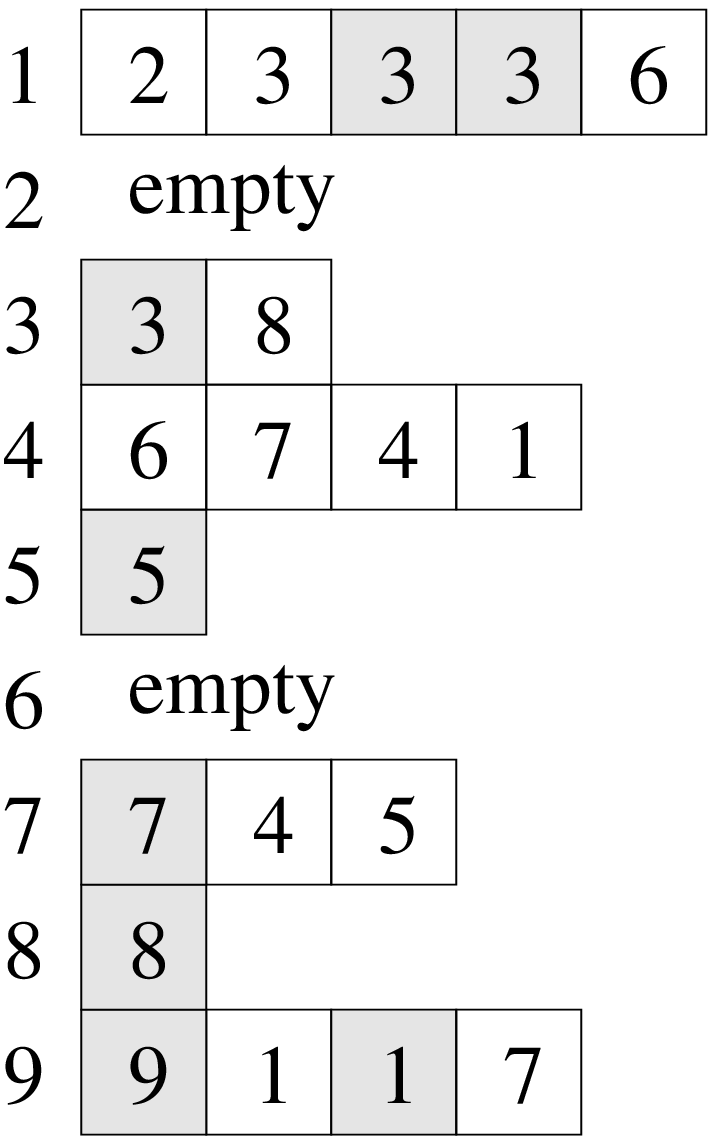}
\caption{Example of a $0$-admissible tableau. A column $0$ has been added and the $0$-critical cells are shaded.}
\end{figure}

Defining
\begin{equation}
d_T^{0}(\alpha) = \prod_{s \, \,  0{\text -{\rm critical}}} d_{\eta}(s)
\, ,
\end{equation}
the combinatorial formula for the non-symmetric Jack polynomials is given by
\begin{equation} \label{nonsymformula}
E_{\eta}(x; \alpha ) = \left(\frac{1}{\prod_{s \in \eta} d_{\eta}(s)} \right)
 \sum_{T \, \,  0{\text -{\rm admissible\ of\ shape\ } \eta}} d_T^{0}(\alpha) \, x^{{\rm ev}(T)} \, ,
\end{equation}
where ${\rm ev}(T)$, the evaluation of $T$, is given by the vector
$(|T|_1,\dots,|T|_N)$ with $|T|_i$ the number of $i$'s in
the $0$-admissible tableau $T$.

\subsection{Jack polynomials in superspace} \label{subsjack}

 Given a superpartition $\Lambda=
(\Lambda_1,\dots,\Lambda_m;\Lambda_{m+1},\dots,\Lambda_{N})$ define
$\tilde \Lambda$ to be the composition
\begin{equation}
\tilde \Lambda := (\Lambda_m,\dots,\Lambda_1,\Lambda_N,\dots,\Lambda_{m+1}) \, .
\end{equation}
It was established in \cite{DLM1} that the Jack polynomials
in superspace can be obtained from the
non-symmetric
Jack polynomials through the following relation:
\begin{equation} \label{jackinnonsym}
{J}_{\Lambda} = \frac{(-1)^{m(m-1)/2}}{f_{\Lambda^s}}
\sum_{w \in S_N} {\mathcal K}_{w} \, \theta_1 \cdots \theta_m \,
E_{\tilde \Lambda}(x;\alpha) \, ,
\end{equation}
where $f_{\Lambda^s}$ was defined in \eqref{equadeff}
and $\mathcal K_w$ was defined at the beginning of the introduction.
In this article, this
will serve as our definition of Jack polynomials in superspace.

Note that the composition $\tilde \Lambda$ is of a very special form.  Its
first $m$ rows (resp. last $N-m$ rows)
are strictly increasing (resp. weakly increasing).
Diagrammatically, it is made of two partitions (the first one of
which without repeated parts)
drawn in the French
notation (largest row in the bottom).
For instance if $\Lambda=(3,1,0;5,3,3,0,0)$, we have
$\tilde \Lambda=(0,1,3,0,0,3,3,5)$ whose diagram is given by
\begin{equation}
{\tableau[scY]
{\bl \bullet \\  \\ && \\ \bl \bullet \\ \bl \bullet \\  & &  \\
 & & \\
& & & &  \\ }}
\end{equation}
We will refer to the first $m$ rows (resp. last $N-m$ rows)
of $\tilde \Lambda$ as the fermionic (resp. non-fermionic)
portion of $\tilde \Lambda$.

\section{The main result}

Given a cell $s$ in $D[\Lambda]$, let
$a_{\La}(s)$ be the number of cells (including the possible circle
at the end of the row) to the right of $s$.  Let also
$\ell_{\La}(s)$ be the number of cells (not including the possible
circle at the bottom of the column) below $s$.  Finally, let
$\Lambda^{\circ}$ be the set of cells of $D[\La]$ that do not
appear at the same time in a row containing a circle {\it and} in a
column containing a circle.  The result we will prove in this article is the
following, which was conjectured in \cite{DLM2}.
\begin{proposition}  \label{propo}
The coefficient  $c_\La^{\mathrm{min}}(\alpha)$ of
$\tilde m_{\Lambda_{\mathrm{min}}}= (\ell_{n,m}!)
m_{\Lambda_{\mathrm{min}}}$ in the monomial expansion of $J_\La$ is given by
\begin{equation}
c_\La^{\mathrm{min}}(\alpha) = \frac{1}{\prod_{s \in
\Lambda^{\circ}} \Bigl( \alpha a_{\Lambda}(s) + \ell_{\Lambda}(s)+1
\Bigr)  } \, .
\end{equation}
\end{proposition}
For instance, in the case $\Lambda=(3,1,0;4,2,1)$,
filling every cell $s \in \La^{\circ}$ with the
corresponding value $\bigl( \alpha a_{\Lambda}(s) + \ell_{\Lambda}(s)+1
\bigr)$, we obtain
\begin{equation} {\small{\tableau[mcY]{{\mbox{\tiny
$3\alpha+5$}}&{\mbox{\tiny $2\alpha+3$}} &{\mbox{\tiny $\alpha+2$}}&
{\mbox{\tiny $1$}} \\& &{\mbox{\tiny
$\alpha+1$}}&\bl\gcercle\\{\mbox{\tiny $\alpha+3$}}&{\mbox{\tiny
$1$}}\\ & \bl\gcercle\\{\mbox{\tiny $1$}}\\  \bl \gcercle}} }
\end{equation}
We thus get in this case
\begin{equation}
c_\La^{\mathrm{min}}(\alpha)
=\frac{1}{(3\alpha+5)(2\alpha+3)(\alpha+2)(\alpha+1)(\alpha+3)} \, .
\end{equation}

\section{Derivation of the identity}

Combining \eqref{nonsymformula} and \eqref{jackinnonsym}, we have
\begin{equation} \label{equajack}
J_{\Lambda} = \frac{(-1)^{m(m-1)/2}}{f_{\Lambda^s}}  \left(\frac{1}{\prod_{s
      \in \tilde \Lambda} d_{\tilde \Lambda}(s)} \right)
\sum_{w \in S_N} {\mathcal K}_{w} \, \theta_1 \cdots \theta_m \,
\sum_{T \, \, 0{\text -{\rm admissible}}} d_T^{0}(\alpha) \, x^{{\rm ev}(T)} \, ,
\end{equation}
where the inner sum is over all $0$-admissible tableaux of shape $\tilde \Lambda$.

To prove Proposition~\ref{propo}, we will compute the
coefficient of
$\tilde m_{\Lambda_{\mathrm{min}}}$ in the R.H.S. of \eqref{equajack}
and show that it is as stated in the proposition.
This will be done in a series of steps that will culminate at the
end of the section with an identity on partitions.  The identity will
then be proven in the next section.

First, it is known \cite{DLM1}
that a given expansion coefficient $c_{\Lambda \Omega}(\alpha)$
in
\begin{equation}
J_{\Lambda} = \sum_\Omega c_{\Lambda \Omega}(\alpha) m_{\Omega}
\end{equation}
does not depend on the number of variables $N$ as long as $N \geq
\ell(\Omega)$.
Therefore, for simplicity we can set $N= \ell_{n,m}+m$ (which corresponds to
 $\ell(\Lambda_{\mathrm{min}})$).  Also, by symmetry, it
is obvious that to compute the coefficient of $m_{\Lambda_{\mathrm{min}}}$ it
suffices to compute the coefficient of
$\theta_1 \cdots \theta_m x^{\Lambda_{\mathrm{min}}}$ in $J_{\Lambda}$.

In the remainder of this article, given a permutation $w$,
${\rm sgn}(w)$ will stand for the sign of the permutation $w$.  Will
will use $S_m$ and $S_{N-m}$ to stand for the subgroups of $S_N$
made out of elements permuting
$\{1,\dots,m\}$ and $\{ m+1,\dots,N\}$ respectively.
\begin{lemma} \label{lemmaev}
We have that $T$ makes a non-zero
contribution to the coefficient of $\theta_1 \cdots \theta_m
x^{\Lambda_{\mathrm{min}}}$ in the R.H.S. of \eqref{equajack}
iff ${\rm ev}(T)=(|T|_1,\dots,|T|_m,1,\dots,1)$ with
$[|T|_1+1,\dots,|T|_m+1]$ a permutation in $S_m$.
Furthermore, when $T$
makes a non-zero contribution we have
$\mathcal K_w \theta_1 \cdots \theta_m x^{{\rm ev}(T)} = \pm \, \theta_1 \cdots \theta_m
x^{\Lambda_{\mathrm{min}}}$,
where $w$ is of the form $w = w_1 \times w_2 \in S_m \times S_{N-m}$ with
$w_1=[m-|T|_1,\dots,m-|T|_m]$, in which case the sign $\pm$ is given
by ${\rm sgn}(w_1)$.
\end{lemma}
\begin{proof}  The first part of the lemma is obvious given that we must
have $\{|T|_1,\dots,|T|_m  \}=\{0,1,\dots,m-1 \}$ for $T$ to make a non-zero
contribution to the coefficient of $\theta_1 \cdots \theta_m
x^{\Lambda_{\mathrm{min}}}$.  The second part follows from the fact that the
permutation $w$ must send $i$ to $m-|T|_i$, for all $i=1,\dots,m$, in
order to have $\mathcal K_w \, x^{{\rm ev}(T)}=x^{\Lambda_{\mathrm{min}}}$.
The sign arises from the anticommutation relations that the $\theta_i$'s
obey.
\end{proof}

Given a tableau $T$, we denote by $T_{(m)}$ the subtableau made out of the
cells of $T$ that
are filled with letters from $\{1,\dots,m \}$.
We say that $P$ is a $\tilde \Lambda$-configuration if
there exists a $T$ that makes a non-zero
contribution to the coefficient of $\theta_1 \cdots \theta_m
x^{\Lambda_{\mathrm{min}}}$ in the R.H.S. of \eqref{equajack} such that
$T_{(m)}=P$. Given a $\tilde \Lambda$-configuration $P$,
we define $\mathcal S_P$ to be the set of $0$-admissible tableaux $T$
such that $T_{(m)}=P$.  We let also
\begin{equation}
d_{P}(\alpha) :=
\prod_{s \, \,  0{\text -{\rm critical}}}
d_{\tilde \Lambda}(s) \, ,
\end{equation}
where a cell $s \in P$ is $0$-critical if it obeys the conditions
(a) or (b) for a $0$-critical cell in a
$0$-admissible tableau. Furthermore, let
$\mathcal C_{\tilde \Lambda}$ be the set of $\tilde \Lambda$-configurations.
\begin{lemma} \label{lemmacoef}
Let $T \in \mathcal S_{P}$ for some $P \in \mathcal C_{\tilde \Lambda}$.
Then
\begin{equation}
d_T^0(\alpha) = d_{P}(\alpha) \prod_{i=N-\ell(\Lambda^s)+1}^N d_{\tilde
  \Lambda}((i,1)) \, .
\end{equation}
\end{lemma}
\begin{proof}
There is exactly one occurrence of the letter $i$
in $T$ for $i=m+1,\dots,N$ (recall that $N=\ell_{n,m}+m$).
By condition (3) of the definition of $0$-admissible tableaux, we
must have a letter $N$ in position $(N,1)$.  Then cell $(N-1,1)$
must be filled with a letter
$N-1$, since letter $N$ has already been used to fill cell $(N,1)$.
Applying this reasoning again
and again we get that position $(i,1)$, for $i=N-\ell(\Lambda^s)+1,\dots,N$,
is filled with a letter $i$.  This implies that all these cells are
$0$-critical and contribute to a factor $\prod_{i=N-\ell(\Lambda^s)+1}^N d_{\tilde
  \Lambda}((i,1))$.
>From the definition of $d_{P}(\alpha)$,
the contribution of the letters $1,\dots,m$ in $d_T^0(\alpha)$
will be
$d_{P}(\alpha)$.
Finally, the remaining letters
$m+1,\dots,N-\ell(\Lambda^s)$ appear
exactly once and cannot occupy positions $(i,1)$ for
$i=m+1,\dots,N-\ell(\Lambda^s)$, since these cells do not belong to $\tilde \Lambda$.
Therefore none of these letters occupies a $0$-critical position in $T$
and thus each of them contributes a factor 1 in $d_T^0(\alpha)$.
\end{proof}
An easy consequence of the proof of the lemma is that the number of
$0$-admissible tableaux in $\mathcal S_{P}$
is equal to $(\ell_{n,m}-\ell(\Lambda^s))!$
for any
$\tilde \Lambda$-configuration $P$.
Using Lemmas~\ref{lemmaev} and \ref{lemmacoef}, and defining
${\rm sgn}(P)$ to be the sign of the permutation
$[m-|P|_1,\dots,m-|P|_m]$, we then get from \eqref{equajack} that
\begin{equation}
J_{\Lambda} \big |_{m_{\Lambda_{\mathrm{min}}}=}
= \, \frac{(-1)^{m(m-1)/2}}{f_{\Lambda^s}}  \left(\frac{
\prod_{i=N-\ell(\Lambda^s)+1}^N d_{\tilde
  \Lambda}((i,1))}{\prod_{s
      \in \tilde \Lambda} d_{\tilde \Lambda}(s)} \right)   (\ell_{n,m}-\ell(\Lambda^s))! \,
\ell_{n,m}!
\sum_{P \in \mathcal C_{\tilde \Lambda} } {\rm sgn}(P)
d_{P}(\alpha) \, ,
\end{equation}
where $\ell_{n,m}!$ accounts for the number of elements in $S_{N-m}$.
The coefficient  $c_\La^{\mathrm{min}}(\alpha)$ of
$\tilde m_{\Lambda_{\mathrm{min}}}= (\ell_{n,m}!)
m_{\Lambda_{\mathrm{min}}}$ in the monomial expansion of $J_\La$ is thus
\begin{equation} \label{equacmin}
c_\La^{\mathrm{min}} (\alpha)=
\frac{(-1)^{m(m-1)/2}}{f_{\Lambda^s}}  \left(\frac{
\prod_{i=N-\ell(\Lambda^s)+1}^N d_{\tilde
  \Lambda}((i,1))}{\prod_{s
      \in \tilde \Lambda} d_{\tilde \Lambda}(s)} \right)
 (\ell_{n,m}-\ell(\Lambda^s))!
\sum_{P \in \mathcal C_{\tilde \Lambda} } {\rm sgn}(P)
d_{P}(\alpha) \, .
\end{equation}
The next lemma will further simplify this equation.
\begin{lemma} \label{lemma2} We have
\begin{equation} \label{eqlemma}
\frac{ \left( \prod_{s
      \in \tilde \Lambda} d_{\tilde \Lambda}(s) \right)
 \left(\prod_{i\geq 1} m_i (\Lambda^s)!
 \right) }{ \prod_{s \in
\Lambda^{\circ}} \Bigl( \alpha a_{\Lambda}(s) + \ell_{\Lambda}(s)+1
\Bigr) }   = \left(\prod_{i=N-\ell(\Lambda^s)+1}^N d_{\tilde
  \Lambda}((i,1))\right)\left(\prod_{1\leq j < i \leq m} d_{\tilde
  \Lambda}((i,\tilde \Lambda_{j}+1))\right)\, .
\end{equation}
\end{lemma}
\begin{figure}
\psfrag{a}{${\tilde a}$}
\psfrag{b}{${\tilde b}$}
\psfrag{c}{${\tilde c}$}
\psfrag{d}{${\tilde d}$}
\psfrag{e}{$a^\circ$}
\psfrag{f}{$d^\circ$}
\psfrag{g}{$b^\circ$}
\psfrag{h}{$c^\circ$}
\includegraphics[height=8cm]{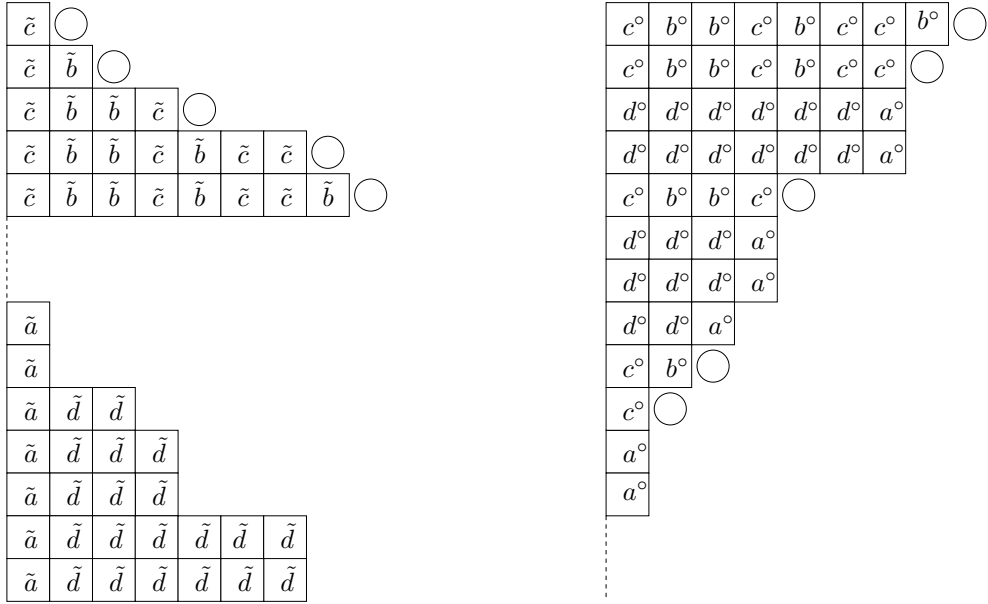}
\caption{\label{fig:simplicationlemma} There is a weight preserving
 bijection between cells of $\{{\tilde c},{\tilde d}\} \subset {\tilde
 \Lambda}$ and those of $\{ c^\circ,d^\circ \} \subset {\Lambda^\circ}
 \subset {\Lambda^*}$. Roughly speaking, this bijection corresponds to
 a sorting of rows according to their length and a cyclic shift of one
 cell to the left for non-fermionic rows. We denote by $W(X)$ the product of
 the appropriate weight of the cells in $X$. The bijection implies
 $W(\{{\tilde c},{\tilde d}\}) = W(\{c^\circ,d^\circ \})$.
This leads to
 $$ \frac{W({\tilde \Lambda})}{W(\{{\tilde a}\})W(\{{\tilde b}\})} =
 W(\{{\tilde c},{\tilde d}\}) = W(\{c^\circ,d^\circ\}) =
 \frac{W(\Lambda^{\circ})}{W(\{a^\circ\})} \, .$$ }
\end{figure}
\begin{proof}  The proof will proceed by
cancellation of certain terms in the L.H.S. of the equation to
obtain the R.H.S. Figure~\ref{fig:simplicationlemma} illustrates the
general idea of the proof.

Suppose $s=(i,j)\in \Lambda^{\circ}$ belongs to a fermionic row of $D[\Lambda]$
(one that ends with a circle).  Then row $i$ of $D[\Lambda]$ corresponds to
a row $k \in \{1,\dots,m\}$ of $\tilde \Lambda$.  We have then
\begin{equation}\label{equal1}
 \alpha \, a_{\Lambda}((i,j)) + \ell_{\Lambda}((i,j))+1 = \alpha(
{\bf a}_{\tilde \Lambda}((k,j))+1) +   {\bf l}_{\tilde \Lambda}((k,j))+1  = d_{\tilde
  \Lambda}((k,j)) \, .
\end{equation}
In this case  $a_{\Lambda}((i,j))= {\bf a}_{\tilde \Lambda}((k,j))+1$ since both rows
are of the same length and row $i$ of $D[\Lambda]$
has a circle (which accounts for the plus one).
We also have that ${\ell}_{\Lambda}((i,j))={\bf l}_{\tilde \Lambda}((k,j))$.
This is because
${\bf l}_{\tilde \Lambda}''((k,j))$ (resp. ${\bf l}_{\tilde \Lambda}'((k,j))$)
accounts  for the
non-fermionic (resp. fermionic) rows that contribute to $\ell_{\Lambda}((i,j))$.
The only way ${\bf l}_{\tilde \Lambda}'$ would not
correspond to
the number of fermionic rows contributing to $\ell_{\Lambda}((i,j))$ is
if some row above row $k$ in the diagram of $\tilde \Lambda$
was of length $j-1$ (in which case it would count
one too many row).  But this is not
possible since this would imply that there is a circle in column $j$ of
$D[\Lambda]$ and thus that $s \not \in \Lambda^{\circ}$.  Therefore
\eqref{equal1} follows.  Note that the cells that are not canceled in the
first $m$ rows of $\tilde \Lambda$ are exactly the cells $(i,\tilde
\Lambda_{j}+1)$, for $1\leq j < i \leq m$, appearing in the R.H.S. of \eqref{eqlemma}.

Suppose $(i,j)\in \Lambda^{\circ}$ does not belong to a fermionic row of
$D[\Lambda]$ and does not lie at the end of its row.  Then row $i$ of $D[\Lambda]$ corresponds to
a row $k \in \{N-\ell(\Lambda^s)+1,\dots,N\}$ of $\tilde \Lambda$.  In this
correspondence, if there are $p$ rows of the same length as row $i$  that do
not end with a circle in $D[\Lambda]$ and row
$i$ is the $r$-th one of them starting from the top, then we choose $k$ to be
also the $r$-th one (also starting from the top) of that length in
the fermionic portion of
$\tilde \Lambda$.
We have then
\begin{equation}\label{equual2}
 \alpha \, a_{\Lambda}((i,j)) + \ell_{\Lambda}((i,j))+1 = \alpha(
{\bf a}_{\tilde \Lambda}((k,j+1))+1) +   {\bf l}_{\tilde \Lambda}((k,j+1))+1  = d_{\tilde
  \Lambda}((k,j+1)) \, .
\end{equation}
It is easy to see that  $a_{\Lambda}((i,j))= {\bf a}_{\tilde \Lambda}((k,j+1))+1$ since both rows
are of the same length and row $i$ of $D[\Lambda]$ is not fermionic.  We now
need to see that ${\ell}_{\Lambda}((i,j))={\bf l}_{\tilde \Lambda}((k,j+1))$.
First,
${\bf l}_{\tilde \Lambda}''((k,j+1))$  accounts for all the rows below row $i$
of $D[\Lambda]$ of the same length as row $i$ and which contribute to
$\ell_{\Lambda}((i,j))$.  Then
${\bf l}_{\tilde \Lambda}'((k,j+1))$
accounts for all the rows below row $i$
of $D[\Lambda]$ smaller than row $i$ that contribute to
$\ell_{\Lambda}((i,j))$.

The cells in the fermionic portion of
$\tilde \Lambda$ that
are not canceled are those that lie in the first column and which correspond
to the cells $(i,1)$, for $i=N-\ell(\Lambda^s)+1,\dots,N$,
appearing in the R.H.S. of \eqref{eqlemma}.  And finally, the cells of
$\Lambda^{\circ}$ that are not canceled are those lying at the end of a
non-fermionic row.  It is easy to see that their contribution is
$\prod_{i\geq 1} m_i (\Lambda^s)!$.
\end{proof}
Using the previous lemma, equation \eqref{equacmin}, and the fact that
\begin{equation}
f_{\Lambda^s} =  (\ell_{n,m}-\ell(\Lambda^s))! \prod_{i\geq 1} m_i
(\Lambda^s)! \, ,
\end{equation}
we have
\begin{equation} \label{equacmin2}
c_\La^{\mathrm{min}} (\alpha)
 \prod_{s \in
\Lambda^{\circ}} \Bigl( \alpha a_{\Lambda}(s) + \ell_{\Lambda}(s)+1
\Bigr)  =
\frac{(-1)^{m(m-1)/2}} {\prod_{1\leq j < i \leq m} d_{\tilde
  \Lambda}((i,\tilde \Lambda_{j}+1))}
\sum_{P \in \mathcal C_{\tilde \Lambda}} {\rm sgn}(P)
d_{P}(\alpha) \, .
\end{equation}
We will now see that it is not necessary to sum over all $P \in \mathcal C_{\tilde \Lambda}$.
Let $\mathcal G_{\tilde \Lambda}$ be the set of all
$\tilde \Lambda$-configurations $P$ such that
for every $i=1,\dots,m$ there is a letter $i$ in column
$j$ of $P$ for $j=1,\dots,|P|_i$.
We will refer to $\mathcal G_{\tilde \Lambda}$ as the set
of good $\tilde \Lambda$-configurations.
\begin{figure}
\psfrag{2}{$2$}
\psfrag{3}{$3$}
\psfrag{b}{$b$}
\psfrag{5}{$5$}
\psfrag{6}{$6$}
\psfrag{a}{$a$}
\psfrag{8}{$8$}
  \includegraphics[height=5cm]{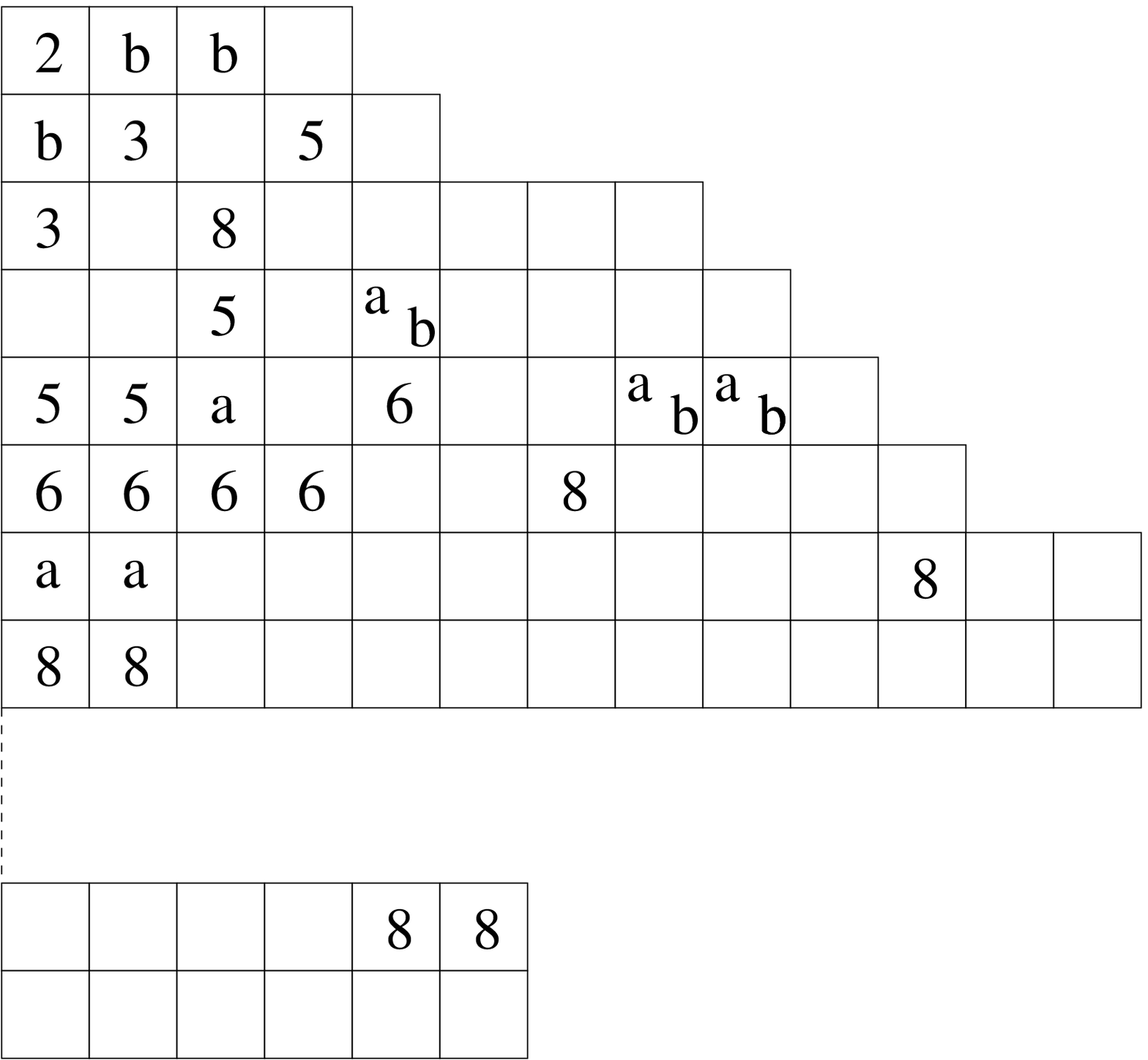}
\caption{\label{fig:involutionemptycolumns} Here are two bad
$\tilde \Lambda$-configurations mapped onto each others by the involution.
Empty cells implicitly contain a label greater than
  $m$ (set equal to 8 in the example).
In cells with two labels, the labels in the upper left (resp. lower right)
corner correspond to the labels of $P$ (resp. $P'$).
In the example, we have $a=7$, $j=4$, $b=4$. Observe that we can have
  labels not larger than $m$ in the non-fermionic portion of a
$\tilde \Lambda$-configuration.
For instance the $8$ in column $5$
is possible only because there is no $8$ in column $4$.}
\end{figure}
\begin{lemma} \label{lemmagconf}
   We have
\begin{equation}
\sum_{ P \in \mathcal C_{\tilde
    \Lambda} } {\rm sgn}(P) \,
d_{P}(\alpha) =\sum_{P \in \mathcal G_{\tilde
    \Lambda}}
{\rm sgn}(P) \,
d_{P}(\alpha)
\end{equation}
\end{lemma}
\begin{proof}  The idea is to construct a sign-reversing involution among the
$\tilde \Lambda$-configurations that do not belong to
$\mathcal G_{\tilde \Lambda}$, which
we will call bad $\tilde
\Lambda$-configurations.   Figure~\ref{fig:involutionemptycolumns}
illustrates the involution that follows.
 Let $P$ be a bad
$\tilde \Lambda$-configuration.
Let $j$ be the smallest integer
such that there exists a letter $a$ that occurs in some column $j'>j$ of $P$
but does not occur in column $j$ of
$P$.
If there are many such $a$'s, pick the one such that
$|P|_a$ is the smallest.
Let $b$ be such that $|P|_b=j-1$.  By definition the $b$'s in
$P$ occur exactly in the first $j-1$ columns.
Therefore $P'$ obtained from $P$ by replacing
the $a$'s that
occur to the right of column $j$ with $b$'s is also
a bad $\tilde \Lambda$-configuration.  We obviously have that ${\rm
  sign}(P')=
- {\rm sgn}(P)$ and $d_{P'}(\alpha)=d_{P}(\alpha)$.
This operation is obviously an involution.
\end{proof}

Now, suppose that  $P$ is a good $\tilde \Lambda$-configuration,
and fix an $i \in \{1,\dots,m \}$.
By the definition of a $0$-admissible tableau (recall that $P=T_{(m)}$ for
some $0$-admissible tableau $T$),
the letter $i$ in the first column of $P$ (if it exists)
is in a row $i_1 \leq i \leq m$.
Again by the the definition of a $0$-admissible tableau,
the letter $i$ in the second column of $P$ (if it exists)
is in a row $i_2 \leq i_i \leq i \leq m$.  Using this argument again and
again, we get that the letters $i$ in column $j=1,\dots,|P|_i$ lie in a
row $i_j$ such that $m \geq i \geq i_1 \geq i_2 \geq \cdots \geq i_{|P|_i}$.
This gives the following lemma.
\begin{lemma} \label{lemmagconfsimple}
 $P$ is a good $\tilde \Lambda$-configuration
iff $[|P|_1+1,\dots,|P|_m+1]$ is a permutation of $S_m$ and the
letters $i$ in column $j=1,\dots,|P|_i$ lie in a
row $i_j$ such that $m \geq i \geq i_1 \geq i_2 \geq \cdots \geq i_{|P|_i}$.
In particular, the cells in a good
$\tilde \Lambda$-configuration
all lie in the first $m$ rows of $\tilde \Lambda$, and thus
the concept of good
$\tilde \Lambda$-configuration only depends on the
fermionic portion of $\tilde \Lambda$.
\end{lemma}

We will now see that there is an easy description of the $\alpha$-hooklengths
of the cells in the fermionic portion of $\tilde \Lambda$.
Let $v_k(\Lambda^s)$
be equal to the number of rows of $\Lambda^s$ that are smaller or
equal to $k$.  Then it is easy to see
that we have, for $(i,j) \in \tilde
\Lambda$
such that $1 \leq i \leq m$:
\begin{equation}\label{eq:lambdasonlambdaa}
d_{\tilde \Lambda}((i,j))= \alpha(\tilde \Lambda_i-j+1)+{\bf l}_{\tilde
  \Lambda}'((m,j)) -(m-i) +v_{\tilde \Lambda_i}(\Lambda^s) - v_{j-1}(\Lambda^s) +1 \, .
 \end{equation}
It proves convenient to write this equation as
\begin{equation} \label{equadsimple}
d_{\tilde \Lambda}((i,j))= a_i+b_j \, ,
 \end{equation}
where $a_i= \alpha \tilde \Lambda_i + v_{\tilde \Lambda_i}(\Lambda^s) +i$
and $b_j = \alpha(1-j)+{\bf l}_{\tilde
  \Lambda}'((m,j)) -m- v_{j-1}(\Lambda^s) +1$.
  Note that we have
\begin{equation}\label{eq:btoa}
b_{\tilde \Lambda_j+1} = 1 -a_j \, ,
\end{equation}
since ${\bf l}_{\tilde \Lambda}'((m,\tilde \Lambda_j+1))=m-j$.  This implies
that
\begin{equation}
(-1)^{m(m-1)/2} \prod_{1\leq j < i \leq m} d_{\tilde
  \Lambda}((i,\tilde \Lambda_{j}+1)) = \prod_{1\leq j < i \leq m} (a_j-a_i-1)
\, .
\end{equation}
Using Lemma~\ref{lemmagconf} and the previous equation, \eqref{equacmin2} becomes
\begin{equation}
c_\La^{\mathrm{min}}(\alpha) \prod_{s \in
\Lambda^{\circ}} \Bigl( \alpha a_{\Lambda}(s) + \ell_{\Lambda}(s)+1
\Bigr)  =
\frac{1} {\prod_{1\leq j < i \leq m} (a_j-a_i-1)}
\sum_{P \in \mathcal G_{\tilde \Lambda}} {\rm sgn}(P)
d_{P} \, ,
\end{equation}
where
\begin{equation} \label{equadp}
d_P  := \prod_{(i,j) \in P \, ; \, (i,j) \, \, 0{\text -{\rm critical}}} (a_i+b_j)
\, .
\end{equation}
First observe that only $b_1,\dots,b_{m-1}$ will appear in $d_P$
since the definition of a good $\tilde \Lambda$-configuration $P$
implies that the cells of $P$ all lie within the first $m-1$ columns, as do all its
$0$-critical cells. It is also
natural to consider the $a_i$'s and $b_i$'s as general
indeterminates rather than as the special expressions given after Equation
\eqref{equadsimple}.    Therefore,
Proposition~\ref{propo} holds if the following identity holds.
\begin{identity}[First form of the identity]  \label{iden1}
Let $a_1,\dots,a_m$ and $b_1,\dots,b_{m-1}$
be indeterminates such that if $\tilde \Lambda_i < m-1$ then
$b_{\tilde \Lambda_i+1} = 1-a_i$.  We have then
\begin{equation}
\prod_{1\leq j < i \leq m} (a_j-a_i-1)  =\sum_{P\in \mathcal G_{\tilde
    \Lambda}}
{\rm sgn}(P) \,
d_{P} \, ,
\end{equation}
where we recall that the sum is over the set of good
$\Lambda_{\mathrm{min}}$-configurations described in
Lemma~\ref{lemmagconfsimple},
${\rm sgn}(P)$ is the sign of the permutation
$[m-|P|_1,\dots,m-|P|_m]$,
and $d_P$ was defined in \eqref{equadp}.
\end{identity}
This identity can be translated into
the language of partitions.
For $i=1,\dots,m$,
let $\lambda^{(i)}$ be a partition of length $i$ with no parts
larger than $m$.
We say that $\lambda^{(1)},\dots,\lambda^{(m)}$ are non-intersecting
if  the $j$-th parts of $\lambda^{(j)},\lambda^{(j+1)},\dots,\lambda^{(m)}$
are distinct for $j=1,\dots,m$.  In particular, this implies that
$[\lambda^{(1)}_1,\dots,\lambda^{(m)}_1]$ is a permutation in $S_m$.
Given $\gamma=(\gamma_1,\dots,\gamma_{m-1})\in \{0,1 \}^{m-1}$,
we define
$\mathcal V_{\gamma}$ to be the set of
$(\lambda^{(1)},\dots,\lambda^{(m)})$ such that
$\lambda^{(1)},\dots,\lambda^{(m)}$ are non-intersecting and
such that
$\lambda^{(i)}_{j+1}> \# \{k \leq j | \, \gamma_k=1  \}$
for all $i=j+1,\dots,m$.  Finally, we say that $(i,j)$ is critical in
$(\lambda^{(1)},\dots,\lambda^{(m)}) \in \mathcal V_{\gamma}$ if
$i \geq j \geq 2$ and $\lambda^{(i)}_j=\lambda^{(i)}_{j-1}$.
\begin{identity}[Second form of the identity] \label{iden2}
Let $\gamma=(\gamma_1,\dots,\gamma_{m-1})\in \{0,1 \}^{m-1}$.
Let also
$a_1,\dots,a_m$ and $b_1,\dots,b_{m-1}$
be indeterminates such that if $\gamma_j = 1$ then
$b_{j} = 1 -a_r$, where $r= \# \{k \leq j | \, \gamma_k=1  \}$.
We have then
\begin{equation}
\prod_{1\leq j < i \leq m} (a_i+1-a_j)
=\sum_{(\lambda^{(1)},\dots,\lambda^{(m)}) \in \mathcal V_{\gamma}}
{\rm sgn}([\lambda_1^{(1)},\dots,\lambda_1^{(m)}]) \,
\prod_{(i,j)~{\rm critical}} (a_{\lambda^{(i)}_j} +
b_{j-1}) \, ,
\end{equation}
where the set $\mathcal V_{\gamma}$ was defined above.
\end{identity}
\begin{proof}[Proof that Identity~\ref{iden1} and Identity~\ref{iden2} are
  equivalent]  Let $\gamma_j=1$ iff there is a part of size $j-1$
in the fermionic portion of $\tilde \Lambda$.
We thus have that $\tilde \Lambda_i< m-1$ iff
 $\gamma_j=1$ for $j=\tilde \Lambda_i+1$.
In this case, $b_{\tilde \Lambda_i+1}=1-a_i$ is equivalent to
$b_{j} = 1 -a_r$, with
$r= \# \{k \leq j | \, \gamma_k=1  \}$, given that $i$ is equal
to the number
of parts smaller or equal to $\Lambda_i$ in the fermionic portion of
$\tilde \Lambda$.  Note that $\gamma$ is only in bijection with the
fermionic portion of $\tilde \Lambda$ whose parts are smaller than
$m-1$.  But since this is the only relevant part in Identity~\ref{iden1},
the relations between the $a_i$'s and $b_j$'s are the same.

\begin{figure}
\psfrag{0}{$0$}
\psfrag{1}{$1$}
\psfrag{2}{$2$}
\psfrag{3}{$3$}
\psfrag{4}{$4$}
\psfrag{5}{$5$}
\psfrag{6}{$6$}
\psfrag{7}{$7$}
\psfrag{m}{$m$}
\psfrag{m-1}{$m-1$}
\psfrag{T}{$P_{\geq m}$}
\psfrag{ga=}{$\gamma=$}
\includegraphics[height=6cm]{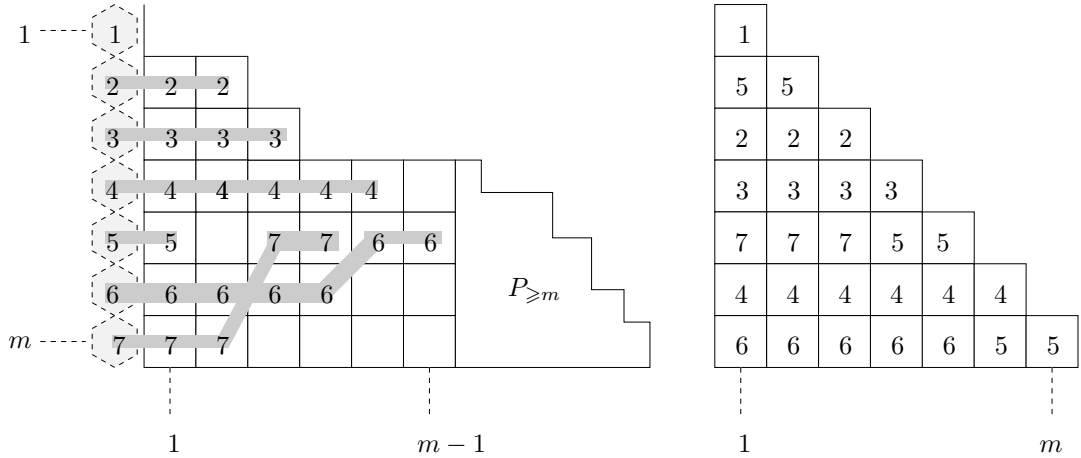}
\caption{\label{fig:fromtableauxtopartitions} An example of
the bijection between $\mathcal G_{\tilde \Lambda}$ and $\mathcal V_{\gamma}$
in the case $m=7$ and $\gamma=(1,0,1,1,0,0)$.
On the left, we draw the diagrammatic representation of the
relevant part of the good-$\tilde \Lambda$ configuration
$P$ and an additional column $0$ of
hexagons labeled by the rows' indices.  Cells to the right of column $m-1$
define a subconfiguration $P_{\geq m}$ whose shape or labels do not
contribute to the weight. On the right, we have the element of $\mathcal V_{\gamma}$
on which is mapped this configuration. In the configuration,
the thick grey line starting from the
hexagon labeled by $i$ represents the row $\lambda^{(k_i+1)} =
(i,i_1,i_2,\ldots i_{k_i})$ in the partition.
 }
\end{figure}

We now show that there is a bijection between the sets $\mathcal G_{\tilde \Lambda}$
and $\mathcal V_{\gamma}$.
Let $P \in \mathcal G_{\tilde \Lambda}$.  Suppose
that letter $i$ is such that $|P_i|=k_i$ (that is, letter $i$
occurs $k_i$ times in $P$).  From Lemma~\ref{lemmagconfsimple},
this implies that letter $i$ appears in columns $1,\dots,k_i$ in
positions  $i_1,\dots,i_{k_i}$ such that
$m \geq i \geq i_1 \geq i_2 \geq \cdots \geq i_{k_i}$.  This gives
us a partition $\lambda^{(k_i+1)}=(i,i_i,i_2,\dots,i_{k_i})$ of length $k_i+1$
with no parts larger than $m$.
If we do the same for $i=1,\dots,m$ we obtain partitions
$\lambda^{(1)},\dots,\lambda^{(m)}$ that are non-intersecting since
the $i_j$'s are
distinct for a fixed $j$ (given that
no two
letters can occupy the same cell). Furthermore,
if
$j> \tilde \Lambda_i$ then cell $(i,j)$ is not in $P$.
The only rows $l$ that are allowed in column $j$ are thus those such that
$l>  \# \{k \leq j | \, \gamma_k=1  \}$.  Since the cell $(l,j)$ in $P$
corresponds to the $(j+1)$-th part of
$\lambda^{(i)}$ for some $i$,
we have the
condition $\lambda^{(i)}_{j+1}> \# \{k \leq j | \, \gamma_k=1  \}$
for all $i=j+1,\dots,m$.  Given
a $(\lambda^{(1)},\dots, \lambda^{(m)}) \in \mathcal V_{\gamma}$, one can
easily reconstruct the corresponding $P \in \mathcal G_{\tilde \Lambda}$
by reversing the procedure we just described.
Figure~\ref{fig:fromtableauxtopartitions}
provides an example of the bijection we just constructed.

If $P \longleftrightarrow (\lambda^{(1)},\dots, \lambda^{(m)})$ in
the bijection, the permutation
 $[\lambda_1^{(1)},\dots,\lambda_1^{(m)}]$ is the inverse of the permutation
$[|P|_1+1,\dots,|P|_m+1]$ since in the bijection $\lambda^{(j)}_1=i$ iff
$|P|_i+1=j$.  This implies that
$${\rm sgn}([\lambda_1^{(1)},\dots,\lambda_1^{(m)}])=
{\rm sgn}([|P|_1+1,\dots,|P|_m+1]) \, ,
$$
given that ${\rm sgn}(w)={\rm sgn}(w^{-1})$ for any permutation $w$. Since
$${\rm sgn}([|P|_1+1,\dots,|P|_m+1])=(-1)^{m(m-1)/2} \,
{\rm sgn}([m-|P|_1,\dots,m-|P|_m]) \, ,$$
this takes into account the changes from $(a_j-a_i-1)$ to $(a_i+1-a_j)$
in the L.H.S. of the identity.

Finally,  we have that
\begin{equation}
\prod_{(i',j') \in P \, ; \, (i',j') \, \, 0{\text -{\rm critical}}} (a_{i'}+b_{j'}) =
\prod_{(i,j)~{\rm critical}} (a_{\lambda^{(i)}_j} +
b_{j-1})\, .
\end{equation}
This is seen in the following way.   Observing that cell $(i',j')$ of $P$,
when
filled with an integer,
corresponds in the bijection to a $\lambda_{j=j'+1}^{(i)}$ for some $i \geq
j$, we have that
$(a_{i'}+b_{j'})=(a_{\lambda_j^{(i)}}+b_{j-1})$. Then
recall that
$(i',j')$ is $0$-critical iff (a)
$j'>1$ and $(i',j'-1)$ is filled with the same letter as $(i',j')$ or (b)
$j'=1$ and $(i',j')=(i',1)$ is filled with an $i'$.  Therefore, we have that
case (a) occurs
iff $\lambda_j^{(i)}=\lambda_{j-1}^{(i)}$ for some $i\geq j\geq 3$ and case (b) occurs
iff $\lambda_2^{(i)}=\lambda_{1}^{(i)}$ for some $i\geq 2$.
\end{proof}


\section{Proof of Identity~\ref{iden2}}

\subsection{Connection with Gessel-Viennot}

We will call the elements in $\mathcal{V}_{\gamma}$ non-intersecting
triangular tableaux compatible with the vector $\gamma$.
The R.H.S. of the equation in Identity~\ref{iden2}
will be denoted by $\Sigma(\gamma)$.
Our goal is thus to show that $\Sigma(\gamma)=\prod_{1\leq j < i \leq m} (a_i+1-a_j)$.

We will say that a partition $\lambda$ of length $i$ is compatible with $\gamma \in \{0,1
\}^{m-1}$ if every part of $\lambda$ is not larger than $m$  and if
$\lambda_{j+1}> \# \{k \leq j | \, \gamma_k=1  \}$
for all $j=1,\dots,\ell(\lambda)-1$.  In this case,
we will say that entry $j$
is critical in $\lambda$ if $\ell(\lambda) \geq j\geq 2$ and $\lambda_{j}=\lambda_{j-1}$.
The weight of $\lambda$ will then simply be
\begin{equation}
w(\lambda) = \prod_{j~{\rm critical}} (a_{\lambda_j} +
b_{j-1}) \, .
\end{equation}
Note that the $a_i$'s and $b_j$'s
are variables not yet necessarily related as in Identity~\ref{iden2}.

We denote by $P_{j,i}(\gamma)$ the sum of weights of
partitions of length $i$, whose first part is equal
to $j$, and that are compatible with $\gamma$ . We define the $m$ by $m$ matrix $M(\gamma)$ as
$$\displaystyle \bigl( M(\gamma) \bigr)_{j,i} = P_{j,i}(\gamma).
$$
A triangular tableau $R$ compatible with the partition $\gamma$ is
 a sequence $(\lambda^{(1)},\ldots,\lambda^{(m)})$ of $m$ partitions compatible
 with $\gamma$ such that $\lambda^{(i)}$ is of length
$i$ and such that
 $\sigma_R = [\lambda^{(1)}_1,\ldots,\lambda^{(m)}_1]$ is a permutation of
 $S_m$.  The weight of a triangular tableau $R$ is
$$\displaystyle
w(R)=\mbox{sign}(\sigma_R)\prod_{i=1}^mw(\lambda^{(i)}).$$
We denote by
$\Sigma_{pi}(\gamma)$ the weighted sum of all the (possibly intersecting)
triangular tableaux compatible with $\gamma$.

\begin{lemma}\label{lgv} For any sequence $\gamma \in \{0,1\}^{m-1}$,
we have
\begin{equation}
\Sigma(\gamma) = \Sigma_{pi}(\gamma) = \det M(\gamma).
\end{equation}
\end{lemma}
For readers familiar with the Lindstr{\"o}m-Gessel-Viennot
lemma (LGV-lemma) \cite{LGV}, remark
that Lemma~\ref{lgv} is an instance of the LGV-lemma.
Indeed, there is an interpretation of Lemma~\ref{lgv} in terms
of ``system of paths'' in a directed acyclic graph depending on
$\gamma$ where each row corresponds to one path. For the sake of
simplicity we choose to reproduce the proof of the general LGV-lemma
in terms of our objects instead of giving an explicit bijection
preserving weights with system of paths of the {\it ad hoc} graph.

\begin{proof}
>From the definition of a determinant, we have
$$\det M(\gamma) = \sum_{\sigma \in S_m}
\mbox{sign}(\sigma)\prod_{i=1}^m P_{\sigma(i),i}(\gamma).
$$
Then, from the definition of $P_{j,i}(\gamma)$, we obtain
$$\det M(\gamma) = \sum_{\sigma \in S_m}
\mbox{sign}(\sigma)\prod_{i=1}^m\left(\sum_{\mbox{$\lambda^{(i)}$ of length $i$
and $\lambda^{(i)}_1=\sigma(i)$}} w(\lambda^{(i)}) \right).$$ After
expanding the product of the $m$ inner sums we recognize the
weighted sum of triangular tableaux compatible with $\gamma$.  Hence
$$ \displaystyle \det M(\gamma) = \Sigma_{pi}(\gamma).$$

\begin{figure}
\psfrag{jr}{$j_R$}
\psfrag{ir}{$i_R$}
\psfrag{kr}{$k_R$}
\psfrag{1}{$1$}
\psfrag{2}{$2$}
\psfrag{3}{$3$}
\psfrag{4}{$4$}
\psfrag{5}{$5$}
\psfrag{6}{$6$}
\psfrag{7}{$7$}
\psfrag{8}{$8$}
\psfrag{9}{$9$}
\psfrag{10}{$10$}
\includegraphics[height=8cm]{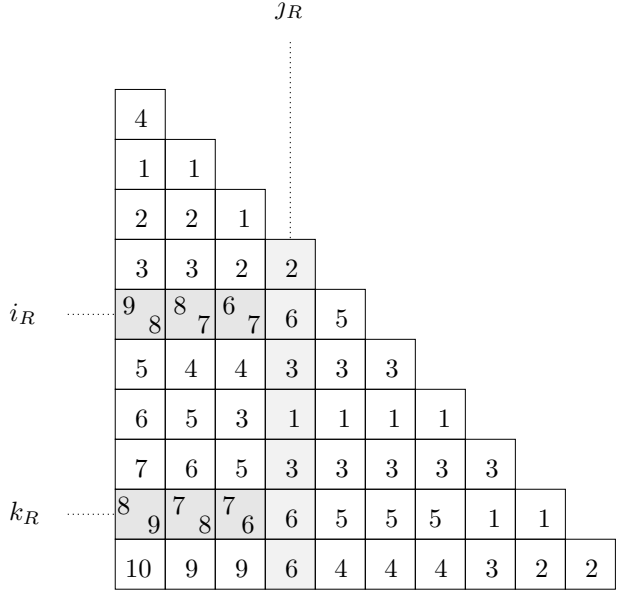}
\caption{The bijection $\Phi$ illustrated with an example. The triangular
  tableaux $R$ and $T$ are represented on the same diagram. Labels of
  $R$ and $T$, when distinct, are in the upper left corner and
  lower right corner respectively.
For the
sake of
  simplicity we chose $\gamma = 0^{m-1}$.}
\end{figure}

We describe a
sign-reversing involution $\Phi$ on the set intersecting tableaux
compatible with $\gamma$ to conclude that $\Sigma(\gamma) =
\Sigma_{pi}(\gamma)$. Let $R = (\lambda^{(1)},\ldots,\lambda^{(m)})$ be such a
tableau. Let $j_R$ be the index of the first column where at least one
entry occurs at least twice. Let $i_R$ be the shortest row in which such
an entry $x_R$ occurs in column $j_R$. Let $k_R$ be the next shortest
row in which $x_R$ occurs in column $j_R$. We define $\Phi(R) = T =
(\tau^{(1)},\ldots,\tau^{(m)})$ by $\tau^{(i_R)}_j=\lambda^{(k_R)}_j$ and
$\tau^{(k_R)}_j=\lambda^{(i_R)}_j$ if $j < j_R$, otherwise $\tau^{(i)}_j
= \lambda^{(i)}_j$.  In other words $\Phi$ corresponds to the exchange of
the entries in row $i_R$ and $k_R$ in all the columns whose index is
strictly lower than $j_R$. Moreover $\Phi$ preserves $j_R$, $i_R$,
$x_R$ and $k_R$ so $\Phi$ is an involution. It remains to check that
$T$ is a triangular tableau compatible with $\gamma$ such that $w(T) =
-w(R)$.  By definition of a triangular tableau, the first column is a
permutation thus $j_R > 1$ so $\sigma_T$ is the appropriate
composition of $\sigma_R$ by the transposition exchanging $i_R$ and
$k_R$. This implies that $\mbox{sign}(\sigma_T) =
-\mbox{sign}(\sigma_R)$. The rows of $T$ remain partitions
because the two exchanged entries in column $j_R-1$
are not smaller than the common value $x_R$ in column $j_R$ of the
corresponding rows.
Finally, it is easy to see that the weight of $R$ and $T$ are
the same.  First observe that by construction the contribution to the weight
coming from the critical entries smaller than $j_R$ is the same in
$\tau^{(i_R)}$ (resp. $\tau^{(k_R)}$) and $\lambda^{(k_R)}$
(resp. $\lambda^{(i_R)}$).  Similarly, the contribution to the weight
coming from the critical entries larger than $j_R$ is the same in
$\tau^{(i_R)}$ (resp. $\tau^{(k_R)}$) and $\lambda^{(i_R)}$
(resp. $\lambda^{(k_R)}$).
The result then follows since the possible critical entry
$j_R$ in $\tau^{(i_R)}$ (resp. $\tau^{(k_R)}$) and in $\lambda^{(k_R)}$
(resp. $\lambda^{(i_R)}$) would give the same contribution to the weight given that
the $j_R$-th entry in both partitions is $x_R$.
\end{proof}

We will first give a proof of the identity in
the case $\ga^0=(0,\ldots,0) \in \{0,1
\}^{m-1}$;
we will do
so by computing the determinant of $M({\ga^0})$ by elementary row
operations using certain technical results that we establish in the
next subsection. From this particular case
we will then be able to prove the result for
an arbitrary $\gamma \in \{0,1\}^{m-1}$.

\subsection{Technical results}

Let $\Par{j}{i}{k}$ be the sum of the weights of all partitions of length $i$
whose first part is $j$ and with at least one part equal to
$j-l$ for each $l=1,\dots,k$; we will use the notation $\PPar{j}{i}{k}$ for this
set of partitions. In particular, we have $\Par{j}{i}{0}=P_{j,i}(\gamma^0)$ which are the
  entries of the matrix $M(\gamma^0)$.

We start with a lemma describing how to compute $\Par{j}{i}{k}$
recursively; we introduce the notation $\Parsh{j}{i}{k}$ to stand for
the result of the substitution $b_1\leftarrow b_2,b_2\leftarrow b_3,
\dots,b_{i-1}\leftarrow b_{i}$
in $\Par{j}{i}{k}$.

\begin{lemma}
 \label{lem:recur}
 Let $k\in \mathbb{N}$. $\Par{j}{i}{k}=0$ if $j\leq k$ or $i\leq
 k$, and $\Par{j}{1}{0}=1$ for $j>0$. Otherwise,
\begin{eqnarray*}
\Par{j}{i}{0}& =&(a_j+b_1)\cdot\Parsh{j}{i-1}{0}+\Parsh{j-1}{i-1}{0}+
[\Par{j-1}{i}{0}-(a_{j-1}+b_1)\Parsh{j-1}{i-1}{0}] \, , \\
\Par{j}{i}{k}& =&(a_j+b_1)\cdot\Parsh{j}{i-1}{k}+\Parsh{j-1}{i-1}{k-1}
 \qquad \text{for~} k\geq 1 \, .
 \end{eqnarray*}
\end{lemma}
\begin{proof}
The first part of the lemma is obvious given that $\PPar{j}{i}{k}$ is empty when
$j\leq k$ or $i\leq k$, and that $\PPar{j}{1}{0}$ contains only the partition $[j]$ of weight $1$.

For the first recurrence formula, the three terms correspond to the
subsets of $\PPar{j}{i}{0}$ consisting of partitions whose
second part has respectively
size $j$, $j-1$, or some $l<j-1$.  This latter term is equal to the
weighted sum of the elements of $\PPar{j-1}{i}{0}$ whose second
part is different from $j-1$.

As for the second recurrence formula, the two terms correspond simply to
the subsets of $\PPar{j}{i}{k}$ made out of partitions whose
second part has respectively
size $j$ and $j-1$.
\end{proof}

Let $\Del{j}{i}{k}$ be the difference $\Par{j}{i}{k}-\Par{j-1}{i}{k}$.
The main result is then the following:

\begin{proposition}
\label{th:interp} For $k\in \mathbb{N}$, $i>k$ and $j>k+1$, we have
$$\Del{j}{i}{k}=(a_j+1-a_{j-k-1})\Par{j}{i}{k+1}$$
\end{proposition}


\begin{proof}
We will prove this relation by induction on $k$.

 Case $k=0$; by reorganizing terms in the first recurrence
formula of Lemma \ref{lem:recur}, we obtain
    $$\Del{j}{i}{0}=(a_j+b_1)\cdot
    \Delsh{j}{i-1}{0}+(a_j+1-a_{j-1})\cdot \Parsh{j-1}{i-1}{0},$$
where $\Delsh{j}{i}{k}$ is naturally defined in general as the result
 of the substitutions $b_l\leftarrow b_{l+1}$ in $\Del{j}{i}{k}$.
  We may assume by induction on $i$, that the case $k=0$ of the proposition
 is true for $\Delsh{j}{i-1}{0}$ (the case $i=1$ being trivial); we
 thus get
 $$\Del{j}{i}{0}=(a_j+1-a_{j-1})\cdot[(a_j+b_1)\Parsh{j}{i-1}{1}
+\Parsh{j-1}{i-1}{0}]\, .$$
Here the second factor on the right hand side is then equal to
$\Par{j}{i}{1}$ by Lemma \ref{lem:recur}.   This proves the proposition in
the case $k=0$.

Case $k>0$; suppose the proposition is true for $k-1$.  This gives
\begin{eqnarray*}
\Del{j}{i}{k}&=&(a_{j}+b_1)\cdot \Delsh{j}{i-1}{k} + (a_j-a_{j-1})\cdot \Parsh{j-1}{i-1}{k}+\Delsh{j-1}{i-1}{k-1} \\
&=& (a_{j}+b_1)(a_j+1-a_{j-k-1})\Parsh{j}{i-1}{k+1}+[(a_j-a_{j-1})+(a_{j-1}+1-a_{j-k-1})]
\cdot \Parsh{j-1}{i-1}{k}\\
&=& (a_j+1-a_{j-k-1})\cdot \left[(a_{j}+b_1)\Parsh{j}{i-1}{k+1} +\Parsh{j-1}{i-1}{k}\right]
\end{eqnarray*}
The first equality comes from Lemma \ref{lem:recur}, the second by
induction on $i$ for $\Delsh{j}{i-1}{k}$ and by the induction
hypothesis for $\Delsh{j-1}{i-1}{k-1}$. We then recognize
$\Par{j}{i}{k+1}$ on the right hand side thanks to Lemma \ref{lem:recur}
again.  The proof is then complete.
\end{proof}

This recursive proof of Proposition \ref{th:interp} does not really explain
the simplicity of its result; for this, we found a bijective proof,
that is given in the Appendix.


\subsection{Proof of the $\gamma^0$ case}
\label{sub:proofsquare}

Let us consider the matrix $M({\ga^0})=(\Par{j}{i}{0})$, whose
determinant we have to compute since, from Lemma~\ref{lgv},
we have $\Sigma(\ga^0)=\det(M({\ga^0}))$.

Let us first perform on $M({\ga^0})$ the elementary row operations
$L_j\leftarrow L_j-L_{j-1}$ with $j=m,m-1,\ldots,2$, in this
order. The coefficients that appear in rows $2$ to $m$ then correspond to
$\Del{j}{i}{0}$ for $j>1$. By Proposition~\ref{th:interp}, we have
that for $j>1$ the quantity $a_j+1-a_{j-1}$ is a factor of every
coefficient in row $j$.

So $\det(M({\ga^0}))=\prod_{j>1}(a_j+1-a_{j-1}) \det(M^{[1]})$, where
the entries of $M^{[1]}$ are given by
\begin{equation*}
m_{ji}^{[1]}=\begin{cases}
&\Par{j}{i}{0}~~\mathrm{ for }~j= 1\\
&\Par{j}{i}{1}~~\mathrm{ for }~j>1
\end{cases}
\end{equation*}

We repeat the operations $L_j\leftarrow L_j-L_{j-1}$ for
$j=m,m-1,\ldots,3$ on $M^{[1]}$.  Coefficients $\Del{j}{i}{1}$
then appear in rows 3 and below.  This implies that the quantities $a_j+1-a_{j-2}$
are factors of the determinant for $j=m,m-1,\ldots,3$.  Factorizing these
quantities  we obtain a new matrix $M^{[2]}$.  One naturally
 applies this process successively,
using Proposition~\ref{th:interp} at each step, to obtain
matrices $M^{[3]},\ldots, M^{[N-1]}$.  At the final stage we get by induction that
\begin{equation}
\det(M({\ga^0}))=\prod_{i>j}(a_i+1-a_j)\times \det(M^{[N-1]}),
\end{equation}
 where the coefficient $(j,i)$ of $M^{[N-1]}$ is
 $m_{ji}^{[N-1]}=\Par{j}{i}{j-1}$.

Now for $j>i$, $\Par{j}{i}{j-1}$ is $0$ by Lemma \ref{lem:recur}; and
for $i=j$, $m_{ii}^{[N-1]}=\Par{i}{i}{i-1}$, which is the weighted
enumeration of $\PPar{i}{i}{i-1}$. But this last set is easily seen to
contain just one element, namely $(i,i-1,\ldots,1)$ which has weight
$1$. So $M^{[N-1]}$ is upper triangular with $1$'s on the diagonal, and
has consequently a determinant equal to $1$.  This completes the proof of
Identity \ref{iden2} in the case of $\ga^0$.


\subsection{Proof of the general case}

We now wish to prove that $\Sigma(\gamma)=\prod_{1\leq j < i \leq m }
(a_i+1-a_j)$ for any
$\ga \in \{0,1 \}^{m-1}$. We will not prove it by using the determinantal form of Lemma \ref{lgv}, but instead by using the $\gamma^0$ case to
deduce all the other cases.\medskip

Let
$i_1<i_2<\ldots<i_k$ be the indices of the 1's in $\gamma$. We will prove the
result by \emph{induction on $k$}.
If $k=0$ then $\ga=\gamma^0$ and the result has already been proven.
 Now let $\gamma$ be a sequence with $k>0$ entries equal to $1$, and
 let $\gamma'$ be the sequence where the last 1 of $\ga$ (with index
 $i_k$) is replaced by a 0. By induction, we have
 $\Sigma(\gamma')=\prod_{i>j}(a_i+1-a_j)$. In particular, the result
does not
 depend on the indeterminate $b_{i_k}$. We may thus set
 $b_{i_k}:=1-a_k$ in
 $\Sigma(\gamma')$
 without changing its value:
 \begin{equation}\label{eqpart1}
\Sigma(\ga')[b_{i_k}:=1-a_k]=\prod_{i>j}(a_i+1-a_j) \, .
\end{equation}
Given that the relations between the $a_i$'s and
$b_j$'s specified by Identity~\ref{iden2} are now satisfied, we
have the following natural decomposition
of weighted sums of non-intersecting triangular
tableaux
\begin{equation}\label{eqpart2}
\Sigma(\ga')[b_{i_k}:=1-a_k]=\Sigma(\gamma)+
\sum_{c\in \mathcal{V}'_{\ga}}w_\ga(c) \, ,\end{equation}
where $\mathcal{V}'_{\ga}$ consists of the non-intersecting triangular
tableaux that are compatible with $\gamma'$ but not with $\gamma$ (observe that
if a tableau is compatible with $\gamma$ then it is compatible with $\gamma'$),
and $w_\ga$ is the weight on tableaux with the relations $b_j=1-a_r$ induced by $\ga$
(as in Identity \ref{iden2}).
Since we wish to prove that $\Sigma(\gamma)=\prod_{i>j}(a_i+1-a_j)$,
we now have to check by equations \eqref{eqpart1} and \eqref{eqpart2} that
the sum over $\mathcal{V}'_{\ga}$ is zero.
This is a consequence of Lemma \ref{lem:invol} in the
next subsection which provides
an involution $\imath$ with no fixed points
on $\mathcal{V}'_{\ga}$  that verifies $w_\ga(c)=-w_\ga(\imath(c))$ for all
$c\in \mathcal{V}'_{\ga}$ when $b_{i_k}:=1-a_k$.
This completes the induction process, and proves Identity \ref{iden2} in all generality.

\subsection{The Involution}

It is easy to check that the set $\mathcal{V}'_{\ga}$ consists of the
non-intersecting triangular
tableaux compatible with $\gamma'$ such that for $j>i_k$
at least one of the
partitions in the tableau has its
$j$-th part equal to $k$. For such a tableau
$c=(\la^{(1)},\ldots,\la^{(m)})$, let $j_{min}$ be the smallest such
$j$, and define $\ell:=j_{min}-1$. Let also $\la^{(r)}$ be the
partition where this $j_{min}$-th part equal to $k$ appears; notice
that $r$ is well defined since from the definition of non-intersecting
triangular tableaux there cannot be two partitions with equal $j_{min}$-th parts.
\medskip

 Now, we define a non-intersecting triangular tableau
 $(\mu^{(1)},\ldots,\mu^{(m)})$ as
  \begin{itemize}
  \item $\mu^{(i)}=\la^{(i)}$ for $i\in \{1,\dots,m \}\backslash\{\ell,r\}$;
  \item $\mu^{(\ell)}=(\la^{(r)}_1,\ldots,\la^{(r)}_{\ell}$);
  \item $\mu^{(r)}=(\la^{(\ell)}_1,\ldots,
      \la^{(\ell)}_\ell,\la^{(r)}_{\ell+1}(=k),\ldots,\la^{(r)}_r)$
\end{itemize}

 For a configuration $c\in\mathcal{V}'_{\ga}$, we then define
  $\imath
(c):=(\mu^{(1)},\ldots,\mu^{(m)})$.

\noindent{\bf Example:~} We illustrate this construction in Figure
\ref{fig:invol}  in the case $\ga=(1,0,1,0,0)$. In the example, $\ell=4$
and $r=6$. The parts $\la^{(i)}_j=2$ with $j>3$ are circled ,
and the entries that are switched are framed.
\begin{figure}
\centering
\psfrag{i}{$\imath$}
\includegraphics{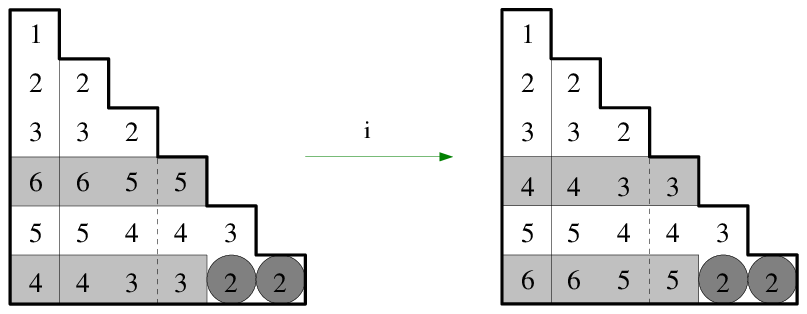}
\caption{The map $\imath$.
\label{fig:invol}}
\end{figure}

\begin{lemma} The map $\imath$ has the following properties:
\begin{itemize}
\item for all $c\in \mathcal{V}'_{\ga}$, we have $\imath(c)\in \mathcal{V}'_{\ga}$;
\item $\imath$ is an involution without fixed points;
\item $w_\ga(c)=-w_\ga(\imath(c))$ for all $c\in \mathcal{V}'_{\ga}$.

\end{itemize}
\label{lem:invol}
\end{lemma}

\begin{proof}
The first two properties are clear from the definition. The signs of $c$
and $\imath(c)$ are opposite since the permutations attached to each
configuration differ by a transposition, namely the one that switches
$\ell$ and $r$.  Finally, one notices immediately that the the contribution
to the weight from the critical
entries are the same as a whole in $c$ and $\imath(c)$
(with possible switches between row $r$ and
$\ell$), except may be for that in position $(r,\ell+1)$.
Since $\la^{(r)}_{\ell+1}=k$, the entry $(r,\ell+1)$ is critical in $c$
or $\imath(c)$ only when
$\la^{(\ell)}_\ell$ or $\la^{(r)}_\ell$ is equal to $k$.
By the minimality of $\ell+1$, we have $\ell=i_k$ in such a case.
The contribution to the weight of this critical entry is thus
$a_k+b_{i_k}=1$ given our choice of specialization.
This completes the proof of the lemma.
\end{proof}

\begin{remark}
 There may be alternative proofs of the results of this section.
First, as observed in Lemma~\ref{lgv}, the quantity $\Sigma(\ga)$ can be
 written as a determinant for any $\ga$, not just for $\ga^0=(0,\ldots,0)$.
 Experimentations using Maple lead us to
 believe that the exact same elementary row operations as those used in the
 case $\ga^0$ can be used to compute the
 determinant in the general case. We did not manage to compute it this
 way, but such a computation might not simplify the
 whole proof anyway.

 A second observation is that the techniques used in
the general case for $\ga$ may actually be used to get rid of the
full computation of the determinant in the $\gamma^0$ case.
For this, it would suffice to \emph{show that $\Sigma(\ga^0)$
is independent of $b_1$}:
 indeed, assuming this is the case, and mimicking the proof in the case of
a general $\ga$, we would get $\Sigma(0^{m-1})=\Sigma(1,0^{m-2})$.
   But the entries of a non-intersecting triangular
 tableaux compatible with $(1,0^{m-2})$ are characterized by $\la^{(1)}=(1)$,
 $\la^{(i)}_1=\la^{(i)}_2$ for $i=2\ldots m$, and $\la^{(i)}_j>1$
 for all $i,j>1$.  From this we easily deduce
 $\Sigma(1,0^{m-2})=\prod_{2\leq i \leq m} (a_i+1-a_1)\times
 \Sigma^\uparrow(0^{m-1})$, where
  $\Sigma^\uparrow(\ga)$ is obtained from $\Sigma(\ga)$ under the
   substitutions $a_i\leftarrow a_{i+1}$. By an immediate induction,
   this would give the desired product for $\Sigma(\ga^0)$.
   Nevertheless, we did not manage to prove the independence from $b_1$
without computing the whole determinant!
\end{remark}

\begin{acknow}  We thank Sylvie Corteel for her interest in
  Identity~\ref{iden20}, and especially for having presented the identity
to PN.
\end{acknow}

\appendix
\section{A bijective proof of Proposition \ref{th:interp}}

We will prove Proposition \ref{th:interp}  bijectively in the following equivalent form:

\begin{proposition}
For $k\in \mathbb{N}$, $i>k$ and $j>k+1$, we have
\begin{equation}
\label{eq:relation_on_P}
(a_j-a_{j-k-1})\Par{j}{i}{k+1} = \left(\Par{j}{i}{k}-\Par{j}{i}{k+1}\right)-\Par{j-1}{i}{k}.
\end{equation}
\end{proposition}

\begin{proof}
 The proof relies on the introduction of a new object:
for $k>0$, an \emph{extended partition} is defined as a
partition $\lambda=(\lambda_1,\ldots,\lambda_i) \in \PPar{j}{i}{k}$
with
a right or left arrow, where the right or left arrow is
located between two successive parts
$\lambda_u$ and $\lambda_{u+1}$ such that
$\la_u>\la_{u+1}=\la_u-1 \geq j-k$.  We say in this case that
$u$ is the position of the arrow
of the extended partition.
For instance, associated to the partition
  $\mu=(6,6,5,5,5,4,2,2,1)\in\PPar{6}{9}{2}$ are the four extensions:
  $$(6,6\overrightarrow{,}5,5,5,4,2,2,1),(6,6\overleftarrow{,}5,5,5,4,2,2,1),
  (6,6,5,5,5\overrightarrow{,}4,2,2,1),(6,6,5,5,5\overleftarrow{,}4,2,2,1) \, ,$$
whose arrows are respectively in positions 2,2,5 and 5.
  We will naturally call left (respectively right) extended partitions those
  with an arrow oriented to the left (resp. to the right), and define
  $\EPPar{j}{i}{k}$ as the set of all extensions of
partitions in  $\PPar{j}{i}{k}$.
The weight of a left (resp. right) extension
of $\lambda$ whose arrow
is in position $u$  is by definition
  the weight of $\lambda$, multiplied by $(a_{\la_u}+b_u)$
  (resp. $-(a_{\la_{u+1}}+b_u)$).
The weights of the four extended partitions above are
  then $w(\mu)$ multiplied respectively by $-(a_5+b_2),(a_6+b_2),-(a_4+b_5)$
  and $(a_5+b_5)$.\medskip

We will now show that both sides of Equation~(\ref{eq:relation_on_P})
are in fact equal to the weighted sum of $\EPPar{j}{i}{k+1}$,
by {\em double counting} this last set.

We consider all the extensions of a given
partition $\lambda \in \PPar{j}{i}{k+1}$. There are clearly $k+1$
left extensions and $k+1$ right extensions of $\la$; if $(u_r)_{r=0\ldots k}$
are the possible positions for the arrows in $\la$, then the weighted sum of
 these $2(k+1)$ extensions is equal to
$$w(\lambda)\left(\sum_{r=0}^{k}(a_{j-r}+b_{u_r})+\sum_{r=0}^{k}-(a_{j-r-1}+b_{u_r})
\right) =w(\lambda)(a_j-a_{j-k-1}) \, .$$
So we obtain indeed the L.H.S. of (\ref{eq:relation_on_P})
 as the total weight of $\EPPar{j}{i}{k+1}$; the proof that it is
also equal to the R.H.S. of (\ref{eq:relation_on_P})
is more involved.

 First, we use
a sign reversing involution $\Psi$ on a certain subset of these extended
partitions. We say that an extended partitions
$\vec{\lambda} \in \EPPar{j}{i}{k+1}$ associated to $\lambda$ (and whose arrow
is in position $u$) is {\em bad}
if one of the following conditions is satisfied:
\begin{enumerate}
\item $\vec{\la}$ is a left extension, and there exists a $v\geq u+1$ such that $\la_v=\la_{v+1}\geq j-k-1$.
\item $\vec{\la}$ is a right extension, and there exists a $v\leq u$ such that
  $\la_{v-1}=\la_v$.
\end{enumerate}

For example, among the four extensions of the partition $\mu$ above,
the first three are bad, and the last one is good (i.e. not bad).
Consider now the following function $\Psi$ on bad extended partitions: if
$\vec{\lambda}$ is
a left extension, choose $v$ minimal in the previous definition;
then $\Psi(\vec{\lambda})$ is defined as
$$(\la_1,\ldots,\la_u,\la_{u+1}+1,\la_{u+2}+1,\ldots,
\la_v+1\overrightarrow{,}\la_{v+1},\ldots,\la_i)$$
And if $\vec{\lambda}$ is a
right extension, choose $v$ maximal in the definition;
$\Psi(\vec{\lambda})$ is then defined as
$$(\la_1,\ldots,\la_{v-1}\overleftarrow{,}\la_v-1,\la_{v+1}-1,\ldots,
\la_u-1,\la_{u+1},\ldots,\la_i)$$

It is then easy to see that $\Psi$ is well defined, is an involution,
and that the weights of $\vec{\la}$ and ${\Psi(\vec{\la})}$ are opposite.
So the weighted sum of $\EPPar{j}{i}{k+1}$ is equal to the sum restricted to
the {\em good} extended partitions,
and we thus need to show that this latter sum is indeed equal to the R.H.S. of
\eqref{eq:relation_on_P}.
 Notice that $\vec{\la} \in \EPPar{j}{i}{k+1}$ is good
iff it is a left extension and there is exactly one part in $\lambda$
of each of the sizes $\la_{u+1},\ldots,j-k-1$,
or it is a right extension and there is exactly one part in $\lambda$
of each of the sizes $j,\ldots,\la_u$.

 There is a bijection $\Theta_L$ between good left extended partitions, and
 partitions in $\PPar{j}{i}{k}$ with at least two equal parts of size
 superior to $j-k-1$, and no part of size $j-k-1$.
  $\Theta_L(\vec{\lambda})$ is obtained from
$\vec{\lambda}$ by deleting the arrow, and increasing by one
the parts $\lambda_{u+1},\ldots,\lambda_v$, where $u$ is the position of the
arrow and $v$ is such that $\lambda_v=j-k-1$. $\Theta_L$ is weight preserving, and the
 weight of its image can be written as
  \begin{equation}
  \label{eqLeft}
  (\Par{j}{i}{k}-\Par{j}{i}{k+1})-L_{j,i}^{[k]} \, ,
  \end{equation}
where ($\Par{j}{i}{k}-\Par{j}{i}{k+1}$) is
 the weight of partitions in $\PPar{j}{i}{k}$
with no part of size $j-k-1$, and $L_{j,i}^{[k]}$ gives the weights of
partitions in $\PPar{j}{i}{k}$ that have exactly one part
of each of the sizes $j,\ldots ,j-k$, and no part of size $j-k-1$.

Then, there is also a bijection $\Theta_R$ between good right extended partitions,
 and partitions in $\PPar{j-1}{i}{k}$ with at least two
equal parts of size between $j-k-1$ and $j-1$.
$\Theta_R(\vec{\lambda})$ is obtained from $\vec{\lambda}$ by deleting the arrow
and by loweringing by one the parts $\lambda_1,\ldots,\lambda_u$,
where $u$ is the position of the arrow.
$\Theta_R$ is weight {\it reversing}, and the weight of its image is
\begin{equation}
  \label{eqRight}
\Par{j-1}{i}{k}-R_{j-1,i}^{[k]} \, ,
\end{equation}
where $R_{j-1,i}^{[k]}$ is the weighted sum of the partitions in
$\PPar{j-1}{i}{k}$ that have exactly one part of each of the sizes
$j-1,\ldots ,j-k-1$.\medskip

 Putting everything together, we have that the weighted sum of
 $\EPPar{j}{i}{k+1}$ is equal to its restriction to good partitions,
 which in turn is equal to \eqref{eqLeft} minus \eqref{eqRight}
 thanks to the weight preserving bijection $\Theta_L$ and
the weight reversing bijection $\Theta_R$.
 But we also have that $R_{j-1,i}^{[k]}=L_{j,i}^{[k]}$ through the weight
 preserving bijection
 that increases by $1$ the first $k$ parts of a partition.
  Thus, we obtain indeed the R.H.S. of
Equation~(\ref{eq:relation_on_P}) as the weighted sum of
 $\EPPar{j}{i}{k+1}$, and the proof is complete.
\end{proof}

\end{document}